\documentclass{amsart}
\usepackage{amsmath,amsfonts,amsthm,amssymb,indentfirst,epic,xurl,graphics,needspace}
\newtheorem{theorem}{Theorem}
\newtheorem{lemma}[theorem]{Lemma}
\newtheorem{corollary}[theorem]{Corollary}
\newtheorem{proposition}[theorem]{Proposition}
\newtheorem{definition}[theorem]{Definition}
\newtheorem{question}[theorem]{Question}
\newtheorem{example}[theorem]{Example}

\newtheorem{convention}[theorem]{Convention}
\numberwithin{equation}{section}
\numberwithin{theorem}{section}
\newcommand{\N}{\mathbb{N}}
\newcommand{\Z}{\mathbb{Z}}

\mathchardef\mhyphen="2D

\begin{document}

\begin{center}
\texttt{Comments, corrections,
and related references welcomed, as always!}\\[.5em]
{\TeX}ed \today
\\[.5em]
\vspace{2em}
\end{center}

\title%
[Prime semigroups and groups]%
{On semigroups that are prime in the sense of Tarski, and\\
groups prime in the senses of Tarski and of Rhodes}
\thanks{%
Archived at \url{http://arxiv.org/abs/2409.15541}\,.
Readable at \url{http://math.berkeley.edu/~gbergman/papers/}.\\
After publication, any updates, errata, related references,
etc., found will be noted at
\url{http://math.berkeley.edu/~gbergman/papers/abstracts}\,.\\
% \hspace*{1.5em}Data sharing not applicable, as no datasets were
% generated or analyzed in this work.
}


\author{George M.\ Bergman}
\address{Department of Mathematics\\
University of California\\
Berkeley, CA 94720-3840, USA}
\email{gbergman@math.berkeley.edu}

\begin{abstract}
If $\mathcal{C}$ is a category of algebras closed under finite
direct products, and $M_\mathcal{C}$ the commutative monoid of
isomorphism classes of members of $\mathcal{C},$ with operation
induced by
direct product, A.\,Tarski defined a nonidentity element $p$ of
$M_\mathcal{C}$ to be {\em prime} if, whenever it divides a
product of two elements, it divides one of them,
and defined an object of $\mathcal{C}$ to be
prime if its isomorphism class has this property.

McKenzie, McNulty and Taylor \cite[p.\,263]{MMT} ask whether
the category of nonempty semigroups has any objects that
are prime in this sense.
We show in \S\ref{S.Main} that it does not.
However, for the category of monoids, and some other subcategories
of semigroups, we obtain examples of prime objects
in~\S\S\ref{S.prep}-\ref{S.N_etc}.
In \S\ref{S.more_MMT_Qs} two related questions
from~\cite{MMT}, open so far as I know, are recalled.

In~\S\ref{S.egg}, which can be read independently of the rest
of this note, we recall two conditions called primeness by
semigroup theorists, and obtain results and examples on the
relationships among those
conditions and Tarski's in categories of groups.
\S\ref{S.Rprime&var} notes a characterization of one
of those conditions on finite algebras in an arbitrary variety.

Several questions are raised.
\end{abstract}
\maketitle
% - - - - - - - - - - - - - - - - - - - - - - - - - - - - - -
\vspace{-.5em}
\section{Note on terminology}\label{S.termin}

As noted in the abstract, Tarski's definition of
``prime'' is different from a pair of
senses currently common in semigroup theory.
In~\S\ref{S.egg} below, the latter two conditions will be recalled,
and their relationship with each other and with Tarski's
on groups examined.
Until that section, {\em prime algebra} will be
understood in Tarski's sense, as indicated
in Definition~\ref{D.gen}(ii) below.

\section{The category of nonempty
semigroups has no prime objects}\label{S.Main}

Here are two general usages that we will follow
in \S\S\ref{S.Main}-\ref{S.more_MMT_Qs}:

\begin{definition}\label{D.gen}
{\rm(i)} As in \cite{MMT}, ``semigroup'' will be
understood to mean ``nonempty semigroup''.
The category of all such semigroups will be denoted $\mathbf{Semigp}.$

{\rm(ii)} Also as in~\cite{MMT} {\rm(}p.\,263, top paragraph{\rm)},
if $\mathcal{C}$ is a category admitting finite direct products
{\rm(}e.g., $\mathbf{Semigp}),$
an object $X$ of $\mathcal{C}$ {\em other than} the final object
{\rm(}the product of the empty family, corresponding to the
identity element of the monoid of isomorphism classes; in
varieties of algebras, the $\!1\!$-element algebra{\rm)}
will be called {\em prime} if,
whenever an object $Y$ of $\mathcal{C}$ admits $X$ as a direct factor,
and $Y = Y_0\times Y_1$ is another direct product decomposition of $Y,$
then one of $Y_0,$ $Y_1$ admits $X$ as a direct factor.
\end{definition}

Here are some further bits of language and notation that will
be used in the present section.

\begin{definition}\label{D.null}
{\rm(i)} By a {\em null semigroup} we shall mean a semigroup
in which all pairs of elements have the same product.
For any cardinal $\kappa,$ the null semigroup of cardinality
$\kappa\!+\!1$ \textup{(}i.e., having exactly $\kappa$ elements
which are not products, in addition to the one which is\textup{)}
will be denoted $\textup{Null}(\kappa).$

{\rm(ii)} Two elements $x$ and $x'$ of a semigroup $S$
will be called {\em action equivalent} if for all $y\in S,$
$x\,y=x'\,y$ and $y\,x=y\,x'.$
\textup{(}Clearly, this is an equivalence relation on $S.)$

{\rm(iii)} An element $x$ of a semigroup $S$ will be called
a {\em product element} if it can be written $x=y\,z$ for $y,z\in S,$
and a {\em non-product element} otherwise.

{\rm(iv)} Given a semigroup $S,$ a {\em skeleton} $S'$ of $S$ will
mean any subsemigroup of $S$ which consists of all the product
elements, and exactly one representative of each
action equivalence class that contains no product elements.
\end{definition}

The definition of ``skeleton'' may look strange; a more natural
characterization is that it is a minimal subsemigroup that
contains at least one element from every action equivalence class.
This guarantees that it contains every product element $x\,y,$
since given $x$ and $y,$ it will contain
some $x'$ action-equivalent to $x$ and
some $y'$ action-equivalent to $y,$ and
hence will contain their product, $x'y'=xy'=xy.$
But since descending chains of nonempty sets can have
empty intersection, the existence of such minimal subsemigroups
is not self-evident, while the existence of $S'$ as in~(iv)
above is.

Clearly,
\begin{equation}\begin{minipage}[c]{35pc}\label{d.skel}
For $S$ a semigroup, all skeleta of $S$ are isomorphic to one another.
\end{minipage}\end{equation}

We also observe that
\begin{equation}\begin{minipage}[c]{35pc}\label{d.null_iff}
A semigroup $S$ is null if and only if it has one and only one
action equivalence class.
\end{minipage}\end{equation}
Indeed, $\Rightarrow$ is clear.
Conversely, if all elements are action
equivalent, then for any $x,\,x',y,\,y'\in S$ we have
$x'y'=x\,y'=x\,y,$ showing that all products are equal.

It is also easy to see that if $S$ and $T$ are semigroups, then
\begin{equation}\begin{minipage}[c]{35pc}\label{d.equiv_in_SxT}
Elements $(s,t),~(s',t')\in S\times T$ are action equivalent
if and only if $s,s'\in S$ are action equivalent and
$t,t'\in T$ are action equivalent.
\end{minipage}\end{equation}
\begin{equation}\begin{minipage}[c]{35pc}\label{d.prod_in_SxT}
An element $(s,t)\in S\times T$ is a product element
if and only if $s\in S$ is a product element and
$t\in T$ is a product element.
\end{minipage}\end{equation}

The next lemma gives lots of isomorphisms among products
of semigroups, which will
be a tool in proving the nonexistence of prime semigroups.

\begin{lemma}\label{L.SxNull}
If $S$ and $T$ are semigroups having isomorphic skeleta, and
$\kappa$ is an infinite cardinal greater than or equal
to the cardinality of every action equivalence class in $S$ and of
every action equivalence class in $T,$ then
$S\times\textup{Null}(\kappa) \cong T\times\textup{Null}(\kappa).$

Moreover, every action equivalence class in that product semigroup
has cardinality $\kappa,$ and contains $\kappa$ non-product elements.
\end{lemma}

\begin{proof}
We can assume without loss of generality that
$S$ and $T$ have a common skeleton; hence, in particular, that
their sets of product elements are the same.
So, in view of~\eqref{d.prod_in_SxT},
$S\times\textup{Null}(\kappa)$ and $T\times\textup{Null}(\kappa)$
have the same product elements.

Since the common skeleton of $S$ and $T$ contains at least one
representative of each action equivalence class, and since the
behavior of the semigroup operation on an element is
determined by which action equivalence class it belongs to,
we can get an isomorphism between
$S\times\textup{Null}(\kappa)$ and $T\times\textup{Null}(\kappa)$
if, for each action equivalence class of the common skeleton
of $S$ and $T,$ we can
define a bijection between the sets of {\em non-product} elements of
the corresponding action equivalence classes of
$S\times\textup{Null}(\kappa)$ and $T\times\textup{Null}(\kappa).$

But each of those non-product sets has cardinality $\kappa.$
E.g., in $S\times\textup{Null}(\kappa),$ such a set
has cardinality {\em at most} $\kappa$ because it is contained in
the direct product of a $\leq\kappa\!$-element action
equivalence class of $S$ and the $\kappa\!+\!1\!$-element
semigroup $\textup{Null}(\kappa),$
and since $\kappa$ is infinite, $\kappa(\kappa\!+\!1)=\kappa;$
while it has {\em at least} $\kappa$
elements occurring as pairs $(s,\alpha),$ where $s$ is any
member of the given equivalence class of $S,$ and
$\alpha$ ranges over the $\kappa$ non-product elements of
$\textup{Null}(\kappa).$
Hence the desired bijections can be chosen.
\end{proof}

We can now prove a large class of cases of our desired result:

\begin{lemma}\label{L.non-null}
No non-null semigroup is prime in $\mathbf{Semigp}.$
\end{lemma}

\begin{proof}
Let $S$ be a non-null semigroup.
Take an infinite cardinal $\kappa$ larger than the
cardinality of any action equivalence class of $S,$
and form a semigroup $S^+\supset S$ by attaching to one
arbitrarily chosen
action equivalence class, $A\subset S,$ $\kappa$ additional elements.
(I.e., add $\kappa$ additional elements to the underlying set of $S,$
and define multiplication on the resulting set so that all products
involving those elements are the same as the products one
gets by putting members of $A$ in their place.
In particular, the product of two of the added elements will
be the same as the product of any two members of $A.)$

By Lemma~\ref{L.SxNull}, we have
\begin{equation}\begin{minipage}[c]{35pc}\label{d.Sxkappa}
$S\times\textup{Null}(\kappa)~\cong~S^+\times\textup{Null}(\kappa).$
\end{minipage}\end{equation}

If $S$ were prime, it would be isomorphic to a direct factor of one
of the factor semigroups on the right-hand side of~\eqref{d.Sxkappa}.
That factor cannot be $\textup{Null}(\kappa),$ because $S,$
being non-null, has more than one action equivalence class.
On the other hand, if it were $S^+,$ so that we could write
\begin{equation}\begin{minipage}[c]{35pc}\label{d.S^+=}
$S^+~\cong~S\times T,$
\end{minipage}\end{equation}
then since $S$ has no action equivalence classes of cardinality
$\kappa,$ but $S^+$ does, the factor $T$ must have one; but the
products of that class with two different action
equivalence classes in $S$ would give two
action equivalence classes in $S^+$ both of cardinality $\kappa,$
contradicting our construction of $S^+.$
So $S$ is not prime.
\end{proof}

To show that semigroups $\textup{Null}(\kappa)$
are also non-prime, we need
a criterion for when a semigroup can be written
as a direct product with $\textup{Null}(\kappa)$ as one of the factors.

\begin{lemma}\label{L.null-kappa_x}
Let $\kappa$ be a cardinal \textup{(}finite or
infinite\textup{)}, and $S$ a semigroup.
Then the following two conditions are equivalent:

{\rm(i)} There exists a semigroup $S_0$ such that
$S~\cong~S_0\times\textup{Null}(\kappa).$

{\rm(ii)} Each action equivalence class $A$ of $S$ can
be written as a disjoint union, $A=A_0\cup A_1,$ such that
$\textup{card}(A_1)=\kappa~\textup{card}(A_0),$ and
all {\em product} elements of $S$ lying in $A$ belong to $A_0.$

Moreover, if $S$ and $\kappa$ are finite,~{\rm(ii)} can be rephrased:

{\rm(ii$\!'\!$)} Each action equivalence class $A$ of $S$ has
cardinality $\kappa\!+\!1$ times an integer greater than or equal
to the number of product elements in~$A.$
\end{lemma}

\begin{proof}
If~(i) holds, let us assume without loss of generality
that $S=S_0\times\textup{Null}(\kappa).$
Note that for each $s\in S_0,$ the $\kappa\!+\!1$
elements $(s,x)$ $(x\in\textup{Null}(\kappa))$ are
action equivalent, and at most one (the one with $x$
the product element of $\textup{Null}(\kappa))$
can be a product element, so taking the union of these families
over the elements $x$ of an action equivalence class of $S_0,$
and writing $A_0$ for the set of $(s,x)$ in this union with $x$ the
product element of $\textup{Null}(\kappa),$ and $A_1$
for the set of $(s,x)$ with non-product $x,$
we get a decomposition of the induced action equivalence
class of~$S$ as in~(ii).

Conversely, if we have a decomposition of each action equivalence
class of $S$ as in~(ii), let $S_0$ be the union
over these classes $A$ of their chosen subsets $A_0.$
This is easily seen to be a subsemigroup of $S$ such that
$S\cong S_0\times\textup{Null}(\kappa).$

When $S$ and $\kappa$ are finite,
the equivalence of~(ii) and~(ii$\!')$ is straightforward.
\end{proof}

As promised, we can now prove

\begin{lemma}\label{L.null_not_prime}
No semigroup $\textup{Null}(\kappa)$ is prime in $\mathbf{Semigp}.$
\end{lemma}

\begin{proof}
For each $\kappa,$ we shall construct a semigroup $S$
which does {\em not} satisfy Lemma~\ref{L.null-kappa_x}(ii), and
a positive integer $n$ such that $\textup{Null}(n)$ also
does not satisfy that condition,
but such that $S\times\textup{Null}(n)$ does; so in
view of Lemma~\ref{L.null-kappa_x}\,(i)\!$\iff$\!(ii), this product
is a counterexample to primeness of $\textup{Null}(\kappa).$

If $\kappa$ is infinite, let $S$ consist of elements
$x_{ij}$ $(i\in\kappa,~j\in\kappa\!+\!1),$ $y_i$ $(i\in\kappa)$
and $z,$ with multiplication given by
\begin{equation}\begin{minipage}[c]{35pc}\label{d.xij}
$x_{ij}\,x_{ij'}~=~y_i$ for $i\in\kappa,~j,j'\in\kappa\!+\!1,$ while
all products not of this form have value~$z.$
\end{minipage}\end{equation}
Since all $\!3\!$-fold products equal~$z,$ $S$ is associative.
Note that the action equivalence classes of $S$ are of two
sorts: on the one hand, the sets
\begin{equation}\begin{minipage}[c]{35pc}\label{d.x-eq-cl}
$\{x_{ij}\,|\,j\in\kappa\!+\!1\},$ one for each $i\in\kappa,$
\end{minipage}\end{equation}
and one more equivalence class, comprising the remaining elements:
\begin{equation}\begin{minipage}[c]{35pc}\label{d.y,z-eq-cl}
$\{y_i\,|\,i\in\kappa\}\,\cup\,\{z\}.$
\end{minipage}\end{equation}
Of these, the classes~\eqref{d.x-eq-cl} clearly satisfy
the condition of Lemma~\ref{L.null-kappa_x}(ii), with $A_0$ the
singleton $\{x_{i0}\},$ and $A_1$ consisting of the remaining $\kappa$
elements, but the class~\eqref{d.y,z-eq-cl} does not,
since all of its elements are products.

Now take any positive integer $n$ (e.g., $n=1).$
Since $\kappa$ is infinite, the finite semigroup $\textup{Null}(n)$
cannot have $\textup{Null}(\kappa)$ as a direct factor.
But in the product semigroup $S\times\textup{Null}(n),$
we have effectively adjoined $n\,\kappa=\kappa=\kappa^2$ non-product
elements $(s,i)$ $(i\in n)$ to each of the $\!\kappa\!$-element
action equivalence classes~\eqref{d.x-eq-cl} and~\eqref{d.y,z-eq-cl}
of $S;$
so the product semigroup satisfies Lemma~\ref{L.null-kappa_x}(ii),
hence has $\textup{Null}(\kappa)$ as a direct factor,
giving the desired example.

Finally, suppose $\kappa$ is finite.

If $\kappa=0,$ then $\textup{Null}(\kappa)$ is the $\!1\!$-element
semigroup, corresponding to the identity element of the monoid of
isomorphism classes, which by definition is not prime.

If $\kappa>0,$ let $S$ be the semigroup of $2(\kappa\!+\!1)$ elements,
\begin{equation}\begin{minipage}[c]{35pc}\label{d.xiyiz}
$x_i$ $(i\in\kappa\!+\!1),$ $y_i$ $(i\in\kappa),$ \,and\, $z,$
\end{minipage}\end{equation}
with multiplication
\begin{equation}\begin{minipage}[c]{35pc}\label{d.x-eq-cl-fin}
$x_i\,x_{i'}~=~y_0,\,$ and all other products equal to\, $z.$
\end{minipage}\end{equation}

Here there are two action equivalence classes, each of
cardinality $\kappa\!+\!1,$ namely $\{x_i\}$ and $\{y_i\}\,\cup\,\{z\}.$
Analogously to the infinite case, we see that
the first satisfies Lemma~\ref{L.null-kappa_x}(ii\!$'$\!),
but the second does not, in this case because it has two
product elements, $y_0$ and $z,$ so that any decomposition as in
Lemma~\ref{L.null-kappa_x}(ii\!$'$\!) would require it to have at least
$2(\kappa\!+\!1)$ elements, while it has only $\kappa\!+\!1.$
So $\textup{Null}(\kappa)$ is not a direct factor of~$S.$

However, if we take the direct product of $S$ with any semigroup
$\textup{Null}(n)$ $(n>0),$ the total number of elements of that
action equivalence class will be multiplied by $n+1,$ while the number
of product elements will not be changed, and we see that
condition~(ii$\!')$ then holds; so $S\times\textup{Null}(n)$
has $\textup{Null}(\kappa)$ as a direct factor.
Moreover, if we choose that $n$ so that $n\!+\!1$ is not a multiple
of $\kappa\!+\!1,$ then $\textup{Null}(n)$ cannot itself have
$\textup{Null}(\kappa)$ as a direct factor, completing the
proof of non-primeness.
\end{proof}

Combining Lemmas~\ref{L.non-null} and~\ref{L.null_not_prime}, we get

\begin{theorem}\label{T.no_primes}
The category $\mathbf{Semigp}$ of \textup{(}nonempty\textup{)}
semigroups has no prime objects.\qed
\end{theorem}

Note that in proving that arbitrary semigroups are not
prime, the auxiliary semigroups that we used -- mainly of
the form $\textup{Null}(\kappa),$ but also a few of other
sorts, e.g., the $S$ in the proof of Lemma~\ref{L.null_not_prime}
-- were all commutative, and had a zero element (an element
$z$ satisfying $zx=xz=z$ for all $x).$

Hence the same reasoning shows that there are no prime
objects in the category of all {\em commutative} semigroups,
or in the category of all semigroups with zero element.
More generally,

\begin{corollary}[to proof]\label{C.cm}
Let $C$ be any full subcategory of $\mathbf{Semigp}$
which is closed under taking pairwise direct products and
direct factors, and contains all semigroups $S$ that have a zero
element $z\in S$ such
that $s_1\,s_2\,s_3=z$ for all $s_1,\,s_2,\,s_3\in S.$

Then $S$ has no prime objects.\qed
\end{corollary}

The above result does not, however, apply with $C$ taken to be the
category of all {\em finite} semigroups, since the
condition that $C$ contain all semigroups $S$ with zero element
equal to all three-fold products is not limited to finite $S.$
Indeed, the proof of
Lemma~\ref{L.non-null} used semigroups $\textup{Null}(\kappa)$ in
a way that required $\kappa$ to be infinite even when the
given semigroup $S$ was finite.
Perhaps there is some variant argument that would use only
finite semigroups in such cases.
Not knowing whether there is, I ask

\begin{question}\label{Q.prime_in_fin}
Does the category of {\em finite} semigroups have any prime objects?
\end{question}

\section{Preparation for some positive results}\label{S.prep}

I obtained the results of the preceding section after a long
grueling attempt to prove that the additive semigroup of
positive integers was prime in $\mathbf{Semigp}.$
Since the argument I eventually found showing that this was not
true makes strong use of action-equivalent elements, it seemed plausible
that positive results might still hold in the category of semigroups
with cancellation.

And, indeed, the arguments I was trying to use do work for
such semigroups,
and can be adapted to several weaker and variant hypotheses.

Some of the preliminary results we will develop
to help us show this are true in more general contexts
than those in which we will use them.
In particular, though we will eventually use
the next lemma with $G$ the additive group of
integers, the proof of the lemma does not require $G$ to be
commutative, so I will use multiplicative notation till we
specialize further.

The referee indicated that %
that lemma was probably known.
The one published result I am now aware of
that implies it is~\cite[Lemma~3.1]{H+R}, where the authors
say it holds by ``easy computation'' but give no further details..
The proof below, suggested by Benjamin Steinberg,
is neater and shorter than my original version.

Recall that an {\em ideal} of a semigroup $S$ means
a nonempty subset $A\subseteq S$ which is closed under
left and right multiplication by elements of~$S.$

\begin{lemma}\label{L.pi_extends}
Let $S$ be a semigroup, $A$ an ideal of $S,$ $G$ a group, and
\begin{equation}\begin{minipage}[c]{35pc}\label{d.pi}
$\pi: A\,\to\,G$
\end{minipage}\end{equation}
a semigroup homomorphism.
Then $\pi$ extends uniquely to a homomorphism $\hat{\pi}: S\to G.$
Indeed, for any $s\in S$ and $x\in A,$ $\hat{\pi}(s)$ can be
described both as $\pi(sx)\,\pi(x)^{-1}$ and as $\pi(x)^{-1}\pi(xs).$
\end{lemma}

\begin{proof}
Choose any $x\in A,$ and define
\begin{equation}\begin{minipage}[c]{35pc}\label{d.hatpi=}
$\hat{\pi}(s)~=~\pi(x)^{-1}\,\pi(x\,s\,x)~\pi(x)^{-1}$ \ for $s\in S.$
\end{minipage}\end{equation}
Since the $x\,s\,x$ factors in $A$ as $x\,s\cdot x$ and $x\cdot s\,x,$
and $\pi$ is a homomorphism on $A,$ we get
corresponding factorizations of $\pi(x\,s\,x).$
Applying these to the right-hand side of~\eqref{d.hatpi=}, and
cancelling the resulting adjacent occurrences of
$\pi(x)$ and $\pi(x)^{-1},$ we get
\begin{equation}\begin{minipage}[c]{35pc}\label{d.hatpi==}
$\hat{\pi}(s)~=~\pi(x)^{-1}\pi(x\,s)~=~\pi(s\,x)\,\pi(x)^{-1}.$
\end{minipage}\end{equation}

To show that $\hat{\pi}:\,S\to G$ is a homomorphism,
let $s,\,t\in S,$ put $s\,t$ in place of $s$ in~\eqref{d.hatpi=},
note that $x\,s\,t\,x=(x\,s)(t\,x)$ with both factors
in $A,$ yielding a factorization of the middle term of
the right-hand side of that case of~\eqref{d.hatpi=}.
Simplifying the result using~\eqref{d.hatpi==}, we get
\begin{equation}\begin{minipage}[c]{35pc}\label{d.pi_xs}
$\hat{\pi}(s\,t)~=~\hat{\pi}(s)\,\hat{\pi}(t)$ \ for all $s,\,t\in A.$
\end{minipage}\end{equation}

It is clear from~\eqref{d.hatpi==} that $\hat{\pi},$ restricted
to~$A,$ gives~$\pi,$ and is the unique homomorphism with this property.
\end{proof}

(The analog of the above lemma with the group $G$ replaced by
a semigroup, or even a monoid, is not true.
For instance, for any $n\geq 1,$ let
$S$ be the additive semigroup of all positive integers,
$A$ the subsemigroup of integers $>n,$
$G$ the additive monoid $A\cup\{0\},$ and
$\pi: A\to G$ the inclusion map.
Likewise, the analog of the lemma with the assumption that $A$ be an
ideal of $S$ weakened to say that it is a left or right ideal fails.
For instance, if $S$ is the free semigroup on two
elements $s$ and $t,$ and we take for $A$ the left ideal
generated by $t,$ then we find that $A$ is free as a
semigroup on the set $\{s^i\,t\,|\,i\geq 0\},$
so for any nontrivial group $G$ there exist homomorphisms
$\pi:A\to G$ acting arbitrarily on these generators.
But such a homomorphism not satisfying
the relations $\pi(s^i\,t)\,\pi(t)^{-1}\pi(s^jt)=\pi(s^{i+j}t)$
for all natural numbers $i$ and $j$ cannot be extended to~$S.)$

\begin{corollary}\label{C.pi_on_prod}
Let $(S_i)_{i\in I}$ be a family of semigroups.

For each $i\in I,$ let $S'_i = S_i \cup \{e\},$ the monoid obtained
by adjoining an identity element to $S.$
Then for any group $G,$ every
semigroup homomorphism $\pi: \prod_I S_i\to G$ has a unique
extension to a monoid homomorphism $\hat{\pi}: \prod_I S'_i\to G.$
\end{corollary}

\begin{proof}
Letting $S = \prod_I S'_i$ and $A=\prod_I S_i,$
it is easy to check that $A$ is an ideal of $S.$
Lemma~\ref{L.pi_extends} now gives the desired conclusion.
\end{proof}

\begin{corollary}\label{C.pi_SxT}
If $S$ and $T$ are semigroups, $G$ a group,
and $\pi: S\times T\to G$ a semigroup homomorphism, then there exist
unique semigroup homomorphisms $\pi_S: S\to G$ and $\pi_T: T\to G$
such that the given homomorphism $\pi$ is described by
\begin{equation}\begin{minipage}[c]{35pc}\label{d.pi(s,t)}
$\pi(s,t)~=~\pi_S(s)~\pi_T(t)$ $(s\in S,\,\,t\in T).$
\end{minipage}\end{equation}
Moreover, $\pi_S(S)$ and $\pi_T(T)$ centralize one another in $G,$
\end{corollary}

\begin{proof}
Apply Corollary~\ref{C.pi_on_prod} with
$I=\{0,1\},$ $S_0 = S,$ $S_1 = T,$
and for $s\in S,$ $t\in T,$ define $\pi_S(s) = \hat{\pi}(s,e)$ and
$\pi_T(t) = \hat{\pi}(e,t).$
From the fact that $\hat{\pi}$ respects
multiplication,~\eqref{d.pi(s,t)} and the mutual centralization
assertion both follow.
Uniqueness can be seen by noting that any $\pi_S$ and $\pi_T$
with the indicated properties extend to
homomorphisms $S\cup\{e\}\to G$ and $T\cup\{e\}\to G,$ and calling on
the uniqueness condition of Corollary~\ref{C.pi_on_prod}.
\end{proof}

(Remark: The corresponding statement with ``{\em monoids}'' in place of
``{\em semigroups}'' is trivial -- just define $\pi_S(s)=\pi(s,e)$
and $\pi_T(t)=\pi(e,t);$ and this even works with $G$ also
assumed a monoid, not necessarily a group.
But for $S$ and $T$ semigroups not assumed to have
neutral elements $e,$ we need the above more
roundabout development, including the assumption that $G$ is a group.)

We now bring in another assumption.
(In applying Lemma~\ref{L.pi_extends} with $G$ an additive
subgroup of the real numbers, we switch
to additive notation for formulas involving $\pi$ and $\hat{\pi}.)$

\begin{lemma}\label{L.pi_nonneg}
In the context of Lemma~\ref{L.pi_extends}, if $G$ is an additive
subgroup of the real numbers, and $\pi: A\to G$ assumes
only nonnegative values, then $\hat{\pi}: S\to G$ also
assumes only nonnegative values.

Hence, likewise, in the contexts of
Corollaries~\ref{C.pi_on_prod} and~\ref{C.pi_SxT}, if $G$ is an
additive subgroup of the reals and $\pi$ assumes only
nonnegative values, then $\hat{\pi},$ respectively
$\pi_S$ and $\pi_T,$ also assume only nonnegative values.
\end{lemma}

\begin{proof}
Suppose, in the context of Lemma~\ref{L.pi_extends}
with $G$ a subgroup of the real numbers,
that some $s\in S$ had $\hat{\pi}(s) < 0.$
Then taking any $x\in A,$ there would be a
positive integer $n$ such that $\pi(s^n x)=n\,\hat{\pi}(s)+\pi(x)<0,$
contradicting our hypothesis on $\pi.$

Applying this in the contexts of
Corollaries~\ref{C.pi_on_prod} and~\ref{C.pi_SxT},
we get the final assertion.
\end{proof}

We are ready to hone in on the class of cases we are interested in.
\vspace{.3em}
\begin{equation}\begin{minipage}[c]{35pc}\label{d.k}
For the remainder of this section,
$k$ will be a fixed nonnegative integer, and
$\N$ the additive semigroup of all nonnegative integers,
so that $\N\!+\!k$ will denote the additive
semigroup of integers~$\geq k.$
\end{minipage}\end{equation}
% \vspace{-.5em}

We can now give necessary and sufficient conditions for
a semigroup to have  $\N\!+\!k$ as a direct factor.

\begin{lemma}\label{L.pi_&_nu}
Let $S$ be a semigroup given with a surjective semigroup
homomorphism $\pi: S\,\to\,\N\!+\!k,$ and a set-map $\nu: S\to S$
satisfying the following three conditions:
\begin{equation}\begin{minipage}[c]{35pc}\label{d.pi_nu(s)}
For all $s\in S,$ $\pi(\nu(s))\,=\,\pi(s) + 1.$
\end{minipage}\end{equation}
\begin{equation}\begin{minipage}[c]{35pc}\label{d.nu_bjctv}
$\nu$ gives a bijection $S\to \{s\in S\,\,|\,\,\pi(s)\,\geq\,k\!+\!1\}.$
\end{minipage}\end{equation}
\begin{equation}\begin{minipage}[c]{35pc}\label{d.nu_st}
For all $s,\,s'\in S,$ $\nu(s)\,s'~=~\nu(s\,s')~=~s\,\,\nu(s').$
\end{minipage}\end{equation}

Then there exists an isomorphism of $S$ with the direct product
$(\N\!+\!k)\times U$ of the additive
semigroup $\N\!+\!k$ and a semigroup
$U,$ such that $\pi$ corresponds to the projection to $\N\!+\!k,$ and
$\nu$ to the map $(i,s)\mapsto (i\!+\!1,s).$

Hence, a semigroup $S$ has a direct product decomposition with
$\N\!+\!k$ as a factor if
and only if it admits maps $\pi$ and $\nu$ as above.
\end{lemma}

\begin{proof}
Given $S,~\pi,$ and $\nu$ as above, note that~\eqref{d.pi_nu(s)}
and~\eqref{d.nu_bjctv} together imply that for every $i\geq 0,$
we have the same statements with $\nu$ replaced by $\nu^i$
and ``$+1$'' replaced by ``$+i$''; and that~\eqref{d.nu_st} likewise
holds with $\nu$ replaced by $\nu^i.$

Now define
\begin{equation}\begin{minipage}[c]{35pc}\label{d.S_0=ker}
$U~=~\{s\in S~|~\pi(s)=k\}.$
\end{minipage}\end{equation}
Since $\pi$ is a semigroup homomorphism, its
value on a product of two elements of $U$ will be $2k,$ hence by
the ``$\!\nu^k\!$-version'' of~\eqref{d.nu_bjctv}
we can define a binary operation on~$U$:
\begin{equation}\begin{minipage}[c]{35pc}\label{d.U_mult}
$u\cdot v~=~\nu^{-k}(u\,v).$
\end{minipage}\end{equation}

From~\eqref{d.nu_st} and
the associativity of $S,$ it is easy to check that for all $u,v,w\in U,$
$(u\cdot v)\cdot w=\nu^{-2k}(u\,v\,w)=u\cdot(v\cdot w),$
so~\eqref{d.U_mult} defines a semigroup structure; and to
verify that as a semigroup, $S$ is isomorphic
to $(\N\!+\!k)\times U$ via the map
\begin{equation}\begin{minipage}[c]{35pc}\label{d.nu^-}
$s~\mapsto~(\pi(s),\ \nu^{-\pi(s)+k}(s)).$
\end{minipage}\end{equation}

This proves the main assertion of the lemma, namely
the ``if'' direction of the final sentence.
The forward implication of that final sentence is straightforward.
\end{proof}

Below, by a {\em direct factor}
of a semigroup $S$ we understand a semigroup $P$ such
that $S$ has a direct product decomposition $S =~ P\times Q.$
In particular, Lemma~\ref{L.pi_&_nu} above characterizes the semigroups
having $\N\!+\!k$ as a direct factor.

\begin{convention}\label{Cv.S,T,nu}
For the remainder of this section we shall assume
that $S$ and $T$ are semigroups such that $S\times T$ has
$\N\!+\!k$ as a direct factor.
We will not write down that factorization,
but simply assume we are given maps $\pi: S\times T\to\N\!+\!k$ and
$\nu: S\times T\to S\times T$ satisfying the conditions of
Lemma~\ref{L.pi_&_nu} with $S\times T$ in the role of $S.$
\end{convention}

In view of Corollary~\ref{C.pi_SxT} and Lemma~\ref{L.pi_nonneg},
the map $\pi$ has the form
\begin{equation}\begin{minipage}[c]{35pc}\label{d.pi+pi}
$\pi(s,t)=\pi_S(s)+\pi_T(t),$ where
$\pi_S:S\to\N$ and $\pi_T:T\to\N$ are semigroup homomorphisms.
\end{minipage}\end{equation}
\vspace{-.5em}

We now turn to the properties of the $\nu$
of Convention~\ref{Cv.S,T,nu}.
For the remainder of this section
\begin{equation}\begin{minipage}[c]{35pc}\label{d.nu(s,t)}
for $(s,t)\in S\times T,$ we shall write
$\nu(s,t) = (\nu_t(s),\,\nu_s(t))\in S\times T,$ and
\end{minipage}\end{equation}
\begin{equation}\begin{minipage}[c]{35pc}\label{d.S'T'}
$S'$ \textup{(}resp.~$T')$ will denote the subsemigroup of
elements of $S$ \textup{(}resp.~$T)$ that can be written as products of
two elements of that semigroup.
\end{minipage}\end{equation}
\vspace{-.5em}

\begin{lemma}\label{L.split_nu}
Assuming Convention~\ref{Cv.S,T,nu} and the
notations~\eqref{d.pi+pi}-\eqref{d.S'T'},
\begin{equation}\begin{minipage}[c]{35pc}\label{d.nu_t1=nu_t2}
for all $s\in S',~t_0,\,t_1\in T',$ one has
$\nu_{t_0}(s) = \nu_{t_1}(s),$
and their common value again lies in $S',$\\[.1em]
and likewise\\[.1em]
for all $s_0,\,s_1\in S',~t\in T',$ one has
$\nu_{s_0}(t) = \nu_{s_1}(t),$
and their common value again lies in $T'.$
\end{minipage}\end{equation}
Hence we will make the simplification of notation
\begin{equation}\begin{minipage}[c]{35pc}\label{d.nu(s)=nu_t(s)}
for $s\in S',$ we shall denote by $\nu(s)$ the common value
of $\nu_t(s)\in S'$ for all $t\in T',$ and\\[.2em]
for $t\in T',$ we shall denote by $\nu(t)$ the common value
of $\nu_s(t)\in T'$ for all $s\in S',$
\end{minipage}\end{equation}
so that
\begin{equation}\begin{minipage}[c]{35pc}\label{d.(nu(s),nu(t))}
for all $s\in S',~t\in T',$ $\nu(s,t) = (\nu(s),\,\nu(t)).$
\end{minipage}\end{equation}
Moreover,
\begin{equation}\begin{minipage}[c]{35pc}\label{d.nu_t(s)s'}
for all $s,\,s'\in S$ and $t\in T,$
$\nu_t(s)\,s'~=~\nu(s\,s')~=~s\,\nu_t(s'),$ and\\[.2em]
for all $s \in S$ and $t,\,t'\in T,$
$\nu_s(t)\,t'~=~\nu(t\,t')~=~t\,\nu_s(t').$
\end{minipage}\end{equation}
\end{lemma}

\begin{proof}
For any $s,\,s'\in S$ and $t,\,t'\in T,$ we have,
by Convention~\ref{Cv.S,T,nu} and~\eqref{d.nu_st},
$\nu(s s',\,t t') = \nu(s, t)\,(s', t').$
Applying~\eqref{d.nu(s,t)} to the first factor on the right-hand side,
this gives $\nu(s s',\,t t') = (\nu_{t}(s)\,s',\,\nu_{s}(t')\,t').$
Taking first components, we have
\begin{equation}\begin{minipage}[c]{35pc}\label{d.nu(t_t')}
$\nu_{t t'}(s\,s')~=~\nu_{t}(s)\,s'.$
\end{minipage}\end{equation}

Note that the right-hand side does not depend on~$t',$
hence neither does the left-hand side; and by the
right-left dual argument, it likewise does not depend on $t,$
so we get the same value on replacing the subscript $t\,t'$ by
any element of $T'.$
Here the argument of $\nu,$ a product of two
arbitrary elements of $S,$ can be any element of $S'.$
This proves the first line of~\eqref{d.nu_t1=nu_t2},
allowing us to adopt the notation in the first line
of~\eqref{d.nu(s)=nu_t(s)}.
The first equality
of the first line of~\eqref{d.nu_t(s)s'} follows; and
by interchanging right and left multiplication, and
the roles of $S$ and $T,$ as appropriate, we
get the remaining equalities
of~\eqref{d.nu_t1=nu_t2} and~\eqref{d.nu_t(s)s'}.
\end{proof}

At this point, we cannot assert~\eqref{d.nu_t1=nu_t2}
without the restriction of the arguments to $S'$ and $T'.$
But we can deduce a strong result on the behavior of
$\pi_S$ and~$\pi_T$ on arbitrary elements:

\begin{lemma}\label{L.S_or_T}
For $S,$ $T,$ $\pi$ and $\nu$ as in Convention~\ref{Cv.S,T,nu}, we have
{\em either}
\begin{equation}\begin{minipage}[c]{35pc}\label{d.pi_S_nu_t}
for all $s\in S,~t\in T,$ $\pi_S(\nu_t(s))~=~\pi_S(s)+1\,$
and\, $\pi_T(\nu_s(t))~=~\pi_T(t),$ {\em or}\\[.3em]
for all $s\in S,~t\in T,$
$\pi_S(\nu_t(s))~=~\pi_S(s)\,$ and\, $\pi_T(\nu_s(t))~=~\pi_T(t)+1.$
\end{minipage}\end{equation}
\end{lemma}

\begin{proof}
Let us fix arbitrary elements
\begin{equation}\begin{minipage}[c]{35pc}\label{d.s0t0}
$s_0\in S',$ $t_0\in T',$
\end{minipage}\end{equation}
and let $m$ and $n$ be the integers such that
\begin{equation}\begin{minipage}[c]{35pc}\label{d.pi(nu(s0))}
$\pi_S(\nu(s_0))~=~\pi_S(s_0)+m$\hspace{.5em} and\hspace{.5em}
$\pi_T(\nu(t_0))~=~\pi_T(t_0)+n.$
\end{minipage}\end{equation}
Applying~\eqref{d.pi_nu(s)} (with $S\times T$ in the role of $S)$
and~\eqref{d.pi+pi} to $(s_0,\,t_0)$ we see that
\begin{equation}\begin{minipage}[c]{35pc}\label{d.m+n=1}
$m + n\,=\,1.$
\end{minipage}\end{equation}

Now note that for any $s\in S,~t\in T,$ we have the following
equalities (using, at the first and third steps, the
fact that $\pi_S$ is a homomorphism, at the
second,~\eqref{d.nu_t(s)s'}, and at the last,~\eqref{d.pi(nu(s0))}):
\begin{equation}\begin{minipage}[c]{35pc}\label{d.pi_S...s0}
$\pi_S(\nu_t(s)) + \pi_S(s_0) = \pi_S(\nu_t(s)\,s_0) =
\pi_S(s\,\nu(s_0))
= \pi_S(s) + \pi_S(\nu(s_0)) = \pi_S(s) + \pi_S(s_0) + m.$
\end{minipage}\end{equation}
Cancelling $\pi_S(s_0)$ at the beginning and end, we get
\begin{equation}\begin{minipage}[c]{35pc}\label{d.pi...+m}
$\pi_S(\nu_t(s))~=~\pi_S(s) + m$ for all $s\in S,$ $t\in T.$
\end{minipage}\end{equation}
Reversing the roles of $S$ and $T,$ we likewise have
\begin{equation}\begin{minipage}[c]{35pc}\label{d.pi...+n}
$\pi_T(\nu_s(t))~=~\pi_T(t) + n$ for all $s\in S,$ $t\in T.$
\end{minipage}\end{equation}

Now by~\eqref{d.pi+pi}, $\pi_S$ assumes only nonnegative values,
so we can choose an $s\in S$ minimizing $\pi_S(s),$
and applying~\eqref{d.pi...+m} to this $s,$
we see that~$m$ must be nonnegative; and the
analogous argument shows the same for~$n.$
From~\eqref{d.m+n=1}, it follows that one of $m$
and $n$ is $0$ and the other $1.$
Substituting into~\eqref{d.pi...+m} and~\eqref{d.pi...+n}
we get~\eqref{d.pi_S_nu_t}.
\end{proof}

\begin{corollary}\label{C.split_pi}
If the first alternative of~\eqref{d.pi_S_nu_t} holds, then
the function $\pi_T$ is identically $0$ and $\pi_S(S)=\N\!+\!k,$
while if the second holds, $\pi_S$ is identically $0$ and
$\pi_T(T)=\N\!+\!k.$
\end{corollary}

\begin{proof}
Assume without loss of generality that
the first alternative of~\eqref{d.pi_S_nu_t} holds,
and suppose $\pi_T$ were not identically zero.
Then it would assume some positive value, and being a semigroup
homomorphism, would assume arbitrarily large values.

Let $s\in S$ be an element at which $\pi_S$ takes on its least value.
By the above observation, we can choose $t\in T$ such that the
value of $\pi(s,\,t)=\pi_S(s)+\pi_T(t)$ is $\geq k\!+\!1.$

Then by~\eqref{d.nu_bjctv} applied to $S\times T,$
this implies that for some $(s',t')\in S\times T,$
$(s,t)=\nu(s',t')=(\nu_{t'}(s'),\nu_{s'}(t'))$
By the case of~\eqref{d.pi_S_nu_t} we are assuming,
this gives $\pi_S(s)=\pi_S(\nu_{t'}(s'))=\pi_S(s')+1,$
so $\pi_S(s')<\pi_S(s),$ contradicting our choice of $s.$

So $\pi_T$ is indeed identically zero.
Substituting this into~\eqref{d.pi+pi} and observing
that $\pi(S\times T)=\N\!+\!k$
by Convention~\ref{Cv.S,T,nu} and the conditions of
Lemma~\ref{L.pi_&_nu} it refers to,
we conclude that $\pi_S(S)=\N\!+\!k,$ as claimed.
\end{proof}

\section{Primeness results, at last}\label{S.N_etc}

We are now ready to note some
subcategories of $\mathbf{Semigp}$ in which we can use the
above results to prove some or all of the semigroups $\N\!+\!k$ prime.
These categories cannot, of course, contain the semigroups
$\textup{Null}(\kappa)$ that were used to prove
{\em non\,}-primeness in~\S\ref{S.Main}.
Here are some restrictions that suffice.

\begin{definition}\label{D.cancellative}
We shall call a semigroup $S$\\[.3em]
{\rm(i)} {\em cancellative} if it satisfies
\begin{equation}\begin{minipage}[c]{35pc}\label{d.rt_cncl}
for all $s\neq s'\in S$ and $s''\in S,$ one has $s\,s''\neq s'\,s'',$
and
\end{minipage}\end{equation}
\vspace{-1.8em}
\begin{equation}\begin{minipage}[c]{35pc}\label{d.lt_cncl}
for all $s\neq s'\in S$ and $s''\in S,$ one has $s''\,s\neq s''\,s',$
\end{minipage}\end{equation}
\vspace{-.8em}\\
{\rm(ii)} {\em right} cancellative if it
merely satisfies~\eqref{d.rt_cncl},
respectively, {\em left} cancellative if it
merely satisfies~\eqref{d.lt_cncl},\\[.3em]
and,\\[.3em]
{\rm(iii)} {\em weakly} cancellative if it satisfies
\begin{equation}\begin{minipage}[c]{35pc}\label{d.wk_cncl}
for all $s\neq s'\in S$ there {\em exists} an $s''\in S$ such
that one of the inequalities $s\,s''\neq s'\,s''$ or
$s''\,s\neq s''\,s'$ holds
\end{minipage}\end{equation}
\textup{(}i.e.~if, in the language of
Definition~\ref{D.null}\,{\rm(ii),}
no two distinct elements of $S$ are action equivalent\textup{)}.
\end{definition}

It is easy to verify

\begin{lemma}\label{L.prods_&_cdns}
For each of conditions~\eqref{d.rt_cncl}, \eqref{d.lt_cncl},
\eqref{d.wk_cncl}, a direct product semigroup $S\times T$ satisfies
the condition in question if and only if both $S$ and $T$ do.\qed
\end{lemma}

We can now prove

\begin{theorem}\label{T.cancellative}
For all nonnegative integers $k,$ the semigroup $\N\!+\!k$ is prime
in the category of cancellative semigroups, and, more generally,
in the categories of right cancellative and left cancellative
semigroups, and still more generally, in the category of
weakly cancellative semigroups.
\end{theorem}

\begin{proof}
We shall prove $\N\!+\!k$ prime in the
category of weakly cancellative semigroups.
In view of Lemma~\ref{L.prods_&_cdns}, the other categories
mentioned are closed in that one under pairwise
direct products and direct factors, hence primeness of $\N\!+\!k$ in
that category will imply primeness in the others.

So suppose $S$ and $T$ are weakly cancellative semigroups such that
$\N\!+\!k$ is a direct factor of $S\times T.$
As shown in  the preceding section, this is equivalent to
the existence of maps $\pi: S\times T\to\N\!+\!k$ and
$\nu: S\times T\to S\times T$
satisfying~\eqref{d.pi_nu(s)}-\eqref{d.nu_st}.
As in \eqref{d.nu(s,t)}, for $(s,t)\in S\times T,$ we shall write
$\nu(s,t) = (\nu_t(s),\,\nu_s(t)).$

We claim that under our present hypotheses, $\nu_t(s)$ must in fact
be a function of $s$ alone, and $\nu_s(t)$ a function of $t$ alone.
By symmetry, it suffices to prove the former statement.

So suppose that for some $s\in S,$ $t,\,t'\in T$ we had
$\nu_t(s)\neq\nu_{t'}(s).$
We now invoke weak cancellativity.
By left-right symmetry with respect to order of
operations, we may assume without loss
of generality that there is some $s'\in S$ such that
$\nu_t(s)\,s'\neq\nu_{t'}(s)\,s'.$
But by \eqref{d.nu_t(s)s'}, both sides equal $\nu(s\,s'),$
a contradiction.

Hence the operators $\nu_t\!: S\to S$ $(t\in T)$
are all equal, and likewise the operators $\nu_s\!: T\to T$ $(s\in S).$
Writing the common values of each as $\nu:S\to S$ and
$\nu:T\to T,$~\eqref{d.nu(s,t)} becomes
\begin{equation}\begin{minipage}[c]{35pc}\label{d.better_nu}
for all $s\in S,$ $t\in T,$ $\nu(s,t) = (\nu(s),\,\nu(t)).$
\end{minipage}\end{equation}

Now by Lemma~\ref{L.S_or_T}, one of the functions $\pi_S,$
$\pi_T$ adds $1$ whenever $\nu$ is applied to its argument;
without loss of generality, let us assume that is $\pi_S.$
Then by Corollary~\ref{C.split_pi}, $\pi_S$ takes on
all the values in $\N\!+\!k$ and $\pi_T$ is identically~$0.$

From the fact that the maps $\pi$ and $\nu$ on $S\times T$
satisfy the conditions~\eqref{d.pi_nu(s)}-\eqref{d.nu_st}
of Lemma~\ref{L.pi_&_nu} (with $S\times T$ in the role of $S),$
we now see that $\pi_S: S\to \N\!+\!k$
and $\nu: S\to S$ satisfy those same conditions,
hence, by that lemma, $\N\!+\!k$ is a direct factor in $S.$
(If, rather, $\pi_T$ is the function that, as indicated
in the preceding paragraph, adds $1$ when $\nu$ is applied
to its argument, then $T$ has $\N\!+\!k$ as a direct factor.)
\end{proof}

Our other result is easier to prove, but applies
only to $\N,$ since for $k>0,$ $\N\!+\!k$ does not belong
to the categories in question.

\begin{theorem}\label{T.monoids&}
The object $\N$ is prime in the full subcategory of
$\mathbf{Semigp}$ whose objects are the monoids, and, more generally,
in the full subcategory of $\mathbf{Semigp}$ whose objects are the
semigroups in which every element is a product.
\end{theorem}

\begin{proof}
In the latter category, Lemma~\ref{L.split_nu} shows that all
the functions $\nu_t: S\to S$ are equal, and likewise
the functions $\nu_s: T\to T.$
This gives us~\eqref{d.better_nu}, and we complete the
proof as for Theorem~\ref{T.cancellative}.
\end{proof}

The monoid case of Theorem~\ref{T.monoids&} is also
implied by the ``weakly cancellative'' case of
Theorem~\ref{T.cancellative}, since multiplication by
the identity element is cancellable.

We can easily go from the above result about monoids
as a full subcategory of $\mathbf{Semigp}$
to one about the category of monoids and {\em monoid} homomorphisms:

\begin{proposition}\label{P.monoids_as_such}
A monoid $S$ which is prime in the full subcategory of $\mathbf{Semigp}$
whose objects are the monoids is also prime in the category of monoids
{\rm(}i.e., the subcategory thereof where morphisms are required
to respect identity elements{\rm)}.
In particular, $\N$ is prime in that category.
\end{proposition}

\begin{proof}[Sketch of proof]
The direct product $S\times T$ of two objects in the category
of monoids clearly has the same semigroup
structure as their direct product in the category of semigroups.
Also, if $S$ and $T$ are semigroups such that
$S\times T$ is a monoid, it is easy to verify that
$S$ and $T$ must be monoids.
Hence a monoid $S$ is a direct factor of a monoid $U$ in the
category of semigroups if and only if it is
a direct factor of $U$ in the category of monoids.

From these observations, it is easy to deduce the first assertion
of this Proposition, and apply it to the monoid~$\N.$
\end{proof}

However, the referee has pointed out that the above property
of~$\N$ is a case of a much more general known result.

The context of that result concerns categories of algebras
whose operations include a binary operation denoted $+$
(but not required to be commutative or even associative) and
a zeroary operation $0,$ such that $\{0\}$ is a subalgebra,
and the identity $0+x=x=x+0$ holds \cite[start of \S5.4]{MMT}.
The result also involves the ``center'' of
such an algebra, a complicated concept developed
in \cite[\S5.5]{MMT}; but all we will need to know about it below is
that for all elements $b$ in the center, there exists an element
$c$ such that $b+c=0$ \cite[Definitions~5.11 and 5.12]{MMT}.
(Warning: the word ``center'' is used in different,
though overlapping ways in other chapters of \cite{MMT}.)
The result is
\begin{equation}\begin{minipage}[c]{35pc}\label{d.MMT_pm}
\cite[Exercise~5.15\,(10), p.\,300]{MMT} \ %
In the category of all algebras of a given type with $0$ and $+$
as above, every nontrivial object with finite center that is
not a direct product of two nontrivial objects is prime.
\end{minipage}\end{equation}

Note that primeness in such a category implies primeness
in any subvariety, and that a monoid with no nontrivial
invertible elements (or more generally, with only finitely
many) necessarily satisfies the ``finite center''
condition of~\eqref{d.MMT_pm}.

Some examples in the variety of monoids that~\eqref{d.MMT_pm}
shows to be prime include: (i)~all nontrivial
submonoids of $\N,$ (ii)~all submonoids $S$ of $\N\times\N$
such that the submonoids $S\cap(\N\times\{0\})$ and
$S\cap(\{0\}\times\N)$ are both nontrivial, and $S$ is strictly
larger than their sum (e.g., $\{(a,\,b)\in\N\times\N\ |
\ a\equiv b\ (\mathrm{mod}~n)\}$ for any $n>1;$
or $\N\times\N\setminus\{(1,0)\}),$
(iii)~all submonoids $S$ of $\N\times\N$ such that
$S\cap(\N\times\{0\})=\{(0,0)\},$ while
$S\cap(\N\times\{1\})$ and $S\cap(\{0\}\times\N)$ each have more
than one element, and (iv)~all {\em free} monoids
on more than one (possibly infinitely many) generators.
(That none of these examples is a nontrivial directly product takes
a bit of thought in case~(ii) and more in case~(iii).
I leave these as exercises for the reader.)

If one is interested in monoids including nontrivial invertible
elements, it is worth noting that the concept of ``center''
used in~\eqref{d.MMT_pm} has the property that every central
element of an algebra is central in the more familiar sense
that it commutes, under $+,$ with every element of the algebra
\cite[p.\,296, sentence containing first display]{MMT}.
This leads to still larger classes of objects with
finite centers in the categories of monoids and of groups.

On the other hand, we shall see in Example~\ref{E.Z} below
that the additive group of
integers is {\em not} prime (in the language of that
section, ``not Tarski-prime'') in the category of groups, from
which it easily follows that it is also non-prime in
the categories of semigroups and monoids considered above.

\section{Two other questions from~\cite{MMT}}\label{S.more_MMT_Qs}

The question answered in \S\ref{S.Main} above is the second
of four questions listed
in~\cite[middle of p.\,263]{MMT}.
It was noted in~\cite{YCor} that the last of those questions,
{\it ``Is $\Z$ prime in the class of all Abelian groups?''}, has
an easy affirmative answer.
So far as I know, the other two remain open.
% DON'T MENTION WIGGINS' RESULT TILL I HEAR FROM HIM OR ...

One of them falls under the general theme of this note.
In it, an {\em idempotent} semigroup means one
satisfying the identity $x\,x=x.$
There is considerable literature on such semigroups,
e.g.~\cite[\S4.4]{JH} and~\cite{JAG}.

\begin{question}\label{Q.idpt}
\cite[p.\,263]{MMT}
Is there any prime in the class of all idempotent semigroups?
\end{question}

The other question has a less obvious connection with this note.
In it, an algebra ``of type $<\!1,1\!>$\!'' means
an algebra having precisely two operations, both unary.

\begin{question}\label{Q.<1,1>}
\cite[p.\,263]{MMT}
Is there any prime in the class of all finite algebras of type
$<\!1,1\!>$ {\rm(}or in any type that has more than one operation of
arity~$\geq 1)$?
\end{question}

An algebra of type $<\!1,1\!>$ can be thought of as a set given with
an action on it of the free monoid on two generators;
so in that way the question has a somewhat similar flavor to those
studied in the preceding sections.

\section{Some results and examples for groups, comparing \\
Tarski-primeness and the two sorts of Rhodes-primeness}\label{S.egg}

We mentioned in \S\ref{S.termin} that two concepts called primeness,
somewhat related to but
distinct from Tarski's, are studied by semigroup theorists.
In this section, we examine the relationships
among these three sorts of primeness when applied to groups.
We give them distinct names in Convention~\ref{Cv.Rh&Ta}(iii)-(v)
below.

This section can be read independent of the preceding material
(ignoring a couple of brief comments about
the contrast between Example~\ref{E.Z}
below and Theorems~\ref{T.cancellative} and~\ref{T.monoids&} above,
and, much later, some of the remarks between
Questions~\ref{Q.fin_vs_inf} and~\ref{Q.smgp,mnd,gp}).

\begin{convention}\label{Cv.Rh&Ta}
In this section, for objects $G,$ $H,$ etc.\ of the
category of all groups, or the category of finite groups,\\[.5em]
{\rm(i)}\,\,\,\, we shall call $H$ a {\em direct factor} of $G$
if $G$ is isomorphic to a direct product $H\times H',$\\[.5em]
{\rm(ii)}\,\, we shall call $H$ a {\em subquotient} of $G$ if $H$
is isomorphic to a homomorphic image of a subgroup of~$G.$\\[-.6em]

Moreover, for $G$ nontrivial,\\[.5em]
{\rm(iii)} we shall call $G$ {\em Tarski-prime} if, whenever
$G$ is a direct factor of a group in the category in
question, and the latter group can also be written
as a direct product $G_0\times G_1,$
then $G$ is a direct factor of one of $G_0$ or~$G_1,$\\[.5em]
{\rm(iv)}\, we shall call $G$ {\em Rhodes-prime with respect
to direct products} if, whenever $G$ is a {\em subquotient}
of a group in the category in question, and the latter group
can be written as a direct product $G_0\times G_1,$
then $G$ is a subquotient of one of $G_0$ or~$G_1$
{\rm(}condition introduced at \cite[p.\,482, line~6]{RS}{\rm);}
and similarly,\\[.5em]
{\rm(v)} we shall call $G$ {\em Rhodes-prime with respect
to semidirect products} if, whenever $G$ is a subquotient
of a group in the category in question, and the latter group
can be written as a {\em semidirect} product $G_0 \rtimes\,G_1,$
then $G$ is a subquotient of one of $G_0$ or~$G_1$
{\rm(}condition introduced at \cite[p.\,228, paragraph~before
Lemma~4.1.31]{RS}{\rm)}.
\end{convention}

(What we name a {\em subquotient} of $G$ in~(ii) above was earlier
also called a {\em factor} of $G,$ e.g.,~\cite[p.\,1 et seq.]{HN},
is called a {\em divisor} of $G$ in~\cite[Def.~1.2.25, p.\,32]{RS},
and is now often called a {\em section} of $G$ by group theorists.
But ``subquotient'' is also fairly common.)
% -- it appears in 980 in MathSciNet reviews

We will also follow the common conventions of
generally writing $e$ for the
identity elements of groups, though for particular groups we may use
notation appropriate for them, e.g., $0$ for
the identity element of the additive group $\Z;$ and of
calling a group $G$ {\em nontrivial} if it has more than one element.
(We use ``nontrivial'' in that sense
in the line introducing~(iii)-(v) above.
In~\cite{RS} the trivial semigroup {\em is} counted as prime
with respect to direct and semidirect products,
but here we have excluded that group to make the comparison
between the Rhodes-primeness conditions
and that of Tarski-primeness more straightforward.)

In view of Theorems~\ref{T.cancellative} and~\ref{T.monoids&}
of the preceding section,
one might expect the additive group $\Z$ to be Tarski-prime
in the category of all groups.
But a key tool in proving those results was the fact that
if two members of $\N$ satisfy $m+n=1$ \eqref{d.m+n=1},
one of them has to be~$0.$
The failure of that statement in~$\Z$ underlies the difference
in its behavior.

\begin{example}\label{E.Z}
In the category of groups, $\Z$ is not Tarski-prime,
but {\em is} Rhodes-prime with respect to semidirect and
direct products.
\end{example}

\begin{proof}
To construct a counterexample to Tarski-primeness, let us choose any
\begin{equation}\begin{minipage}[c]{35pc}\label{d.p,q}
prime numbers $p$ and $q,$ and nonunit divisors $a\,|\,p\,{-}1$
and $b\,|\,q\,{-}1,$ such that $a$ and $b$ are relatively prime.
\end{minipage}\end{equation}
(E.g., $p=3,$ $q=7,$ $a=2,$ $b=3.$
Or one could take both $p$ and $q$ to be $7,$
again with $a=2$ and $b=3.)$

Recalling that the multiplicative groups of the rings
$\Z_p$ and $\Z_q$ are cyclic of orders $p-1$ and $q-1,$ we can
\begin{equation}\begin{minipage}[c]{35pc}\label{d.G}
let $G=(\Z_p\rtimes\Z)\times(\Z_q\rtimes\Z)$ where, in the first
factor, a generator of $\Z$ acts on $\Z_p$ by multiplication by
a unit of the ring $\Z_p$ having multiplicative order $a,$
and in the second, such a generator acts on $\Z_q$ by multiplication by
a unit of the ring $\Z_q$ having multiplicative order $b.$
\end{minipage}\end{equation}

Note that neither of the groups $\Z_p\rtimes\Z$ and $\Z_q\rtimes\Z$
has a direct product decomposition with a factor isomorphic to $\Z.$
E.g., if $\Z_p\rtimes\Z$ did, then a generator of that
factor would be a central element of infinite order, hence
would have, under the given description of the group,
the form $(0,\,ma)$ $(m\neq 0).$
But that would make it an $\!a\!$-th power in the whole
group, while looking at the $\!\Z\!$-component in such a direct
product decomposition, we see that a generator thereof can't be an
$\!a\!$-th power for any $a>1.$
(The groups $\Z_p\rtimes\Z$ and $\Z_q\rtimes\Z$ are
in fact directly indecomposable, but the above argument
gives as much as we need.)

On the other hand, suppose we
rewrite $G$ as $(\Z_p\times\Z_q)\rtimes(\Z\times\Z),$
using the action of $\Z\times\Z$ under which the first factor $\Z$
affects only the first factor of $\Z_p\times\Z_q,$ and the second
affects only the second.

Since $a$ and $b$ are relatively prime,
the subgroup of $\Z\times\Z$ generated by the element $(a,b)$
is a direct factor thereof.
Explicitly, we can find integers $c$ and $d$ such that $bc+ad=1,$
and use the right action of the invertible matrix
$(\begin{smallmatrix} \ b & d \\ \!\!-a & c \end{smallmatrix})$
to get an isomorphism $\Z\times\Z\cong\Z\times\Z$ such that
the image of $(a,b)$ generates the second factor.
(That matrix or its inverse necessarily has
some negative entries, which can be thought of as the way
the fact that $\Z$ includes negative integers comes in.)
We see from~\eqref{d.G} that under the action
on $\Z_p\times\Z_q,$ that second factor acts trivially.

Thus, $G$ can be written as a semidirect product
$G=(\Z_p\times\Z_q)\rtimes(\Z\times\Z)$ where the second
$\Z$ acts trivially,
% on $\Z_p\times\Z_q,$
and all the action comes from the first.
Hence we can rewrite this group as a {\em direct} product
$G=((\Z_p\times \Z_q)\rtimes\Z)\times\Z.$
So the group $\Z$ is isomorphic to a direct factor of the
product group~\eqref{d.G}, though it is
not a direct factor in either of the given factors.
So $\Z$ is not Tarski-prime.

As to the Rhodes-primeness conditions, it is clear
that $\Z$ is a subquotient
of a group $G$ if and only if $G$ has an element of infinite order.
It is also easy to see that a direct or semidirect
product $G_0\times G_1$ or $G_0\rtimes G_1$ has an element
if infinite order if and only if $G_0$ or $G_1$ does.
So $\Z$ satisfies both versions of Rhodes-primeness.
\end{proof}

For finite groups, the relationship between these
primeness conditions is quite different.
A key result is the {\em Krull-Schmidt Theorem}
\cite[Theorem~3.8, p.\,86]{TWH}, which says that every
group with ACC and DCC on normal subgroups has a factorization
as a direct product of finitely many nontrivial groups which are not
themselves nontrivial direct products, and that the factors in such
a decomposition are {\em unique} up to rearrangement and isomorphism.
Thus, every group with those chain conditions which is {\em not} a
nontrivial direct product is
Tarski-prime in the category of such groups.
In particular, every nontrivial {\em finite} group which is not
a nontrivial direct product is Tarski-prime in the
category of finite groups.
(For a result about finite algebras of
much more general sorts than groups, of which this is a special case,
see~\eqref{d.MMT_pm} above,
together with the paragraph that precedes that display,
and the second paragraph following it.)

Recall that if a group $G$ has a {\em unique minimal}
nontrivial normal subgroup $M,$ then $G$ is called
{\em monolithic}, and $M$ is called the {\em monolith} of~$G$
\cite[sentence beginning at bottom of p.\,146]{HN}.
We can now prove

\begin{lemma}\label{L.=>=>}
In the category of {\em finite groups,} for $G$ a nontrivial group
we have the implications
\begin{equation}\begin{minipage}[c]{35pc}\label{d.=>=>}
\hspace{2.4em}$G$ is Rhodes-prime with respect to semidirect products\\
$\implies$ $G$ is Rhodes-prime with respect to direct products\\
$\implies$ $G$ is monolithic\\
$\implies$ $G$ is Tarski-prime.
\end{minipage}\end{equation}
\end{lemma}

\begin{proof}
The first implication is clear from the definitions, since
direct products are a special case of semidirect products.

We shall prove the remaining two implications by contradiction.
For the middle implication, suppose $G$ has two minimal
nontrivial normal subgroups $N\neq N'.$
Then $N\cap N'=\{e\},$
so the natural map $G\to (G/N)\times(G/N')$ is an embedding.
Thus $G$ is isomorphic to a subgroup of the
direct product $(G/N)\times(G/N'),$ making it a subquotient thereof;
but it is not a subquotient of
either $G/N$ or $G/N',$ since these have smaller orders than~$G.$
So $G$ is not Rhodes-prime with respect to direct products.

For the final implication, note that by the Krull-Schmidt
Theorem, every non-Tarski-prime nontrivial group $G$ is a direct
product $G_0\times G_1$ of nontrivial groups, and if we choose minimal
nontrivial normal subgroups $N_0\subseteq G_0$ and $N_1\subseteq G_1,$
then $N_0\times\{e\}$ and $\{e\}\times N_1$ are distinct
minimal nontrivial normal subgroups of $G,$ so $G$ is not monolithic.
\end{proof}

Now for some positive results on when
finite groups are Rhodes-prime, in each of the two senses; though
our first result will not be limited to finite groups:

\begin{proposition}\label{P.x|-prime}
Every nontrivial {\em simple} group $G$ {\rm(}commutative
or noncommutative{\rm)} is Rhodes-prime with
respect to semidirect products in the category of all groups.

Hence every nontrivial {\em finite} simple group $G$ is Rhodes-prime
with respect to semidirect products in the category of finite groups.
\end{proposition}

\begin{proof}
Let $G$ be a nontrivial simple group, and suppose
$G$ is a subquotient of a semidirect product $G_0\rtimes\,G_1.$
That is, suppose that that we can write it as $f(H)$ for some subgroup
$H<G_0\rtimes\,G_1,$ and some surjective homomorphism $f:H\to G.$

Since $G_0\rtimes\{e\}$ is a normal subgroup of $G_0\rtimes G_1,$
$H\cap(G_0\rtimes\{e\})$ will be a normal subgroup of $H,$
so its image in $G$ under $f$ will be normal; so
since $G$ is simple, that image will either be $G$ or $\{e\}.$

If $f(H\cap(G_0\rtimes\{e\}))=G,$ that makes $G$ a homomorphic
image of a subgroup of~$G_0,$ i.e., a subquotient of $G_0.$

On the other hand, if $f(H\cap(G_0\rtimes\{e\}))=\{e\},$
we see that the image in $G$ of an element
$(g_0,g_1)\in H$ depends only on $g_1,$ and easily deduce that
$G$ is a homomorphic image of the subgroup of $G_1$ consisting
elements occurring as second components of members of $H;$
so $G$ is a subquotient of $G_1.$

This proves the Rhodes-primeness of $G$ with
respect to semidirect products.
If $G$ is finite, having that property in the category of
all groups clearly implies that it has the same property in the
category of finite groups.
\end{proof}

As noted in Example~\ref{E.Z}, the converse to
Proposition~\ref{P.x|-prime} fails in the category of all groups:
the non-simple group $\Z$ is Rhodes-prime with
respect to semidirect products.
But Alexander Olshanskiy (personal communication) has provided
a proof that
that converse does hold in the category of finite groups.
Here is a simplified version of his proof.

\begin{proposition}[A.\,Olshanskiy]\label{P.AO_simple}
If a nontrivial finite group $G$ is Rhodes-prime with
respect to semidirect products in the category of
finite groups, then $G$ is simple.
\end{proposition}

\begin{proof}
Let $N$ be a minimal nontrivial normal subgroup of~$G.$
(By Lemma~\ref{L.=>=>} there is a
unique such $N,$ but we will not need to call on that fact.)

Being an extension of $N$ by $G/N,$
the group $G$ can be written as a subgroup of the wreath
product $N\wr G/N,$ that is, of the semidirect product
$N^{|G/N|}\rtimes G/N,$ where $|G/N|$ denotes the
underlying set of $G/N,$ and $G/N$ acts on that direct power
of $N$ by permutation of the factors
(Theorem of Kaloujnine and Krasner, \cite[Theorem~22.21, p.\,46]{HN}).
% refers to Marc Krasner and L\'{e}o Kaloujnine
% Produit complet des groupes de permutations et probl\`{e}me
% d'extension de groupes. III. (French)
% Acta Sci. Math. (Szeged) 14 (1951), 69-82.
% MR0049892
% But description in MR would take a lot of thought to translate
% to wreath product; and seems limited to case of faithful action
Hence by Rhodes-primeness with respect to semidirect products,
$G$ must be a subquotient of $G/N$ or of~$N^{|G/N|}.$
The former is impossible because $G/N$ has smaller
order than $G,$ so $G$ must be a subquotient of~$N^{|G/N|}.$

Though $N$ is a minimal nontrivial normal subgroup of~$G,$
it need not be simple, but
it will be {\em characteristically simple,} i.e.,
it will have no proper nontrivial subgroup invariant
under all its automorphisms (since these include conjugation
by all members of $G).$
Hence by~\cite[3.3.15,~p.\,87, and sentence at end of
proof,~p.\,88]{DJRS},
% {DJRS} gives no ref
$N$ is a direct product of (mutually isomorphic) simple groups.
Hence $N^{|G/N|}$ is a direct product of those same simple groups,
with possibly more repetitions.

Thus $G,$ being Rhodes-prime with respect to semidirect products,
and hence in particular, with respect to direct products, must be
a homomorphic image of a subgroup one of those simple groups,~$S.$
(Our definition of that Rhodes-primeness condition only
refers to pairwise products, but by induction, it
extends to arbitrary finite products.)
But since $S$ is a direct factor of the subgroup~$N$
of~$G,$ its order is less than or equal to that of~$G,$
so to be a homomorphic image of a subgroup of~$S,$ the group~$G$
must, in fact, be isomorphic to $S,$ hence, as desired, simple.
\end{proof}

The next result gives some sufficient conditions for $G$ to be
Rhodes-prime with respect to {\em direct} products.
In view of the second implication of~\eqref{d.=>=>}, $G$
must be monolithic.
Given this, we find that some quite varied additional
conditions imply the desired Rhodes-primeness.
(The referee points out that case~(i) below is
\mbox{\cite[Theorem~10.1]{RF+RMcK}} applied to the variety of groups.)

\begin{proposition}\label{P.x-prime}
Let $G$ be a monolithic finite group, with monolith~$M.$

Then $G$ is Rhodes-prime with respect to direct products
in the category of finite groups if any
of the following conditions holds:

\hspace{.5em}{\rm(i)} $M$ is noncommutative, or

\hspace{.25em}{\rm(ii)} $G$ is a semidirect product
$M\rtimes K$ of $M$ with a subgroup $K<G,$ or

{\rm(iii)} $G$ is a cyclic group of prime-power order, $\Z_{p^n}.$
\end{proposition}

\begin{proof}
Given a finite monolithic
group $G$ that is {\em not} Rhodes-prime with respect to
direct products, we shall show that neither~(i) nor~(ii) can hold.
The proof that~(iii) implies Rhodes-primeness will be more
straightforward.

To get the first two results, suppose $G$
is not Rhodes-prime with respect to direct products, i.e.,
can be written as a homomorphic image $f(H)$ of a subgroup $H$ of
a direct product $G_0\times G_1$ of finite groups, such that
$G$ is not a subquotient of $G_0$ or of $G_1.$
We can clearly replace $G_0$ and $G_1$ by
the subgroups given by the projections of $H$ onto those two factors,
so assume that those projections are
surjective; i.e., that $H$ is a subdirect product of $G_0$ and $G_1.$
Let us further assume that, for the given
group $G,$ the groups $G_0,$ $G_1,$ and $H$ are chosen so
as to minimize the order of $H.$

Now let
\begin{equation}\begin{minipage}[c]{35pc}\label{d.H12}
$H_0 = \{h\in G_0\,|\,(h,e)\in H\}$\hspace{.8em}and\hspace{.8em}\
$H_1 = \{h\in G_1\,|\,(e,h)\in H\}.$
\end{minipage}\end{equation}

If $H_0$ were trivial, then
no two elements of $H$ would fall together under the
projection $G_0\times G_1\to G_1,$ so $G_1$ would be
isomorphic to $H,$ so as $G$ is a homomorphic image
of $H,$ this would contradict the assumption that $G$
was not a subquotient of $G_1.$
So $H_0$ is nontrivial; and similarly $H_1.$
From our assumption that $H$ projects surjectively to each of
$G_0$ and $G_1$ it is also easy to see
that $H_0$ is normal in $G_0,$ and $H_1$ in $G_1.$
(E.g., given $h\in H_0$ and $g_0\in G_0,$ we can find $g_1\in G_1$
such that $(g_0,g_1)\in H;$ and conjugating $(h,e)$ by $(g_0,g_1),$ we
effectively conjugate $h$ by $g_0.)$
Now if $\mathrm{ker}(f)$ had
nontrivial intersection with the subgroup $H_0\times\{e\}$ of $H,$ then
since $H_0\times\{e\}$ and $\mathrm{ker}(f)$ are both normal in $H,$
their intersection would be normal, and dividing $G_0$ by the
projection of this intersection, we could decrease the order of $H.$
This, and the corresponding observation for $\{e\}\times H_1,$ give
\begin{equation}\begin{minipage}[c]{35pc}\label{d.triv_cap}
$\mathrm{ker}(f)$ has trivial intersections with $H_0\times\{e\}$
and with $\{e\}\times H_1.$
\end{minipage}\end{equation}

Now let $M_0\times\{e\}$ be a nontrivial subgroup of
$H_0\times\{e\}$ minimal for being normal in $H.$
Since $M_0\subseteq H_0,$~\eqref{d.triv_cap} tells us
that $M_0\times\{e\}$ maps one-to-one into $G,$ hence its
image is a minimal normal subgroup of $G;$ hence that image is~$M.$
Combining with the corresponding observation about
a minimal normal subgroup $\{e\}\times M_1$ of
$\{e\}\times H_1,$ we get
\begin{equation}\begin{minipage}[c]{35pc}\label{d.M1,M2}
$H_0\times \{e\}$ is monolithic with monolith $M_0\times\{e\},$
and $\{e\}\times H_1$ is monolithic with monolith
$\{e\}\times M_1,$ and both of these map
isomorphically to $M$ under $f.$
\end{minipage}\end{equation}

But $M_0\times\{e\}$ and $\{e\}\times M_1$
centralize one another in $H,$ so $M$ must be
self-centralizing in $G,$ i.e., commutative,
giving the desired contradiction to~(i).

To get the next assertion, suppose as in~(ii)
that $G$ is a semidirect product $M\rtimes K$ for some $K<G.$
Since $K$ has trivial intersection with the monolith $M$ of $G,$
\begin{equation}\begin{minipage}[c]{35pc}\label{d.K_has_no_nm}
$K$ contains no nontrivial normal subgroup of~$G.$
\end{minipage}\end{equation}
On the other hand, since $M$ is invariant under conjugation by
members of $K,$ the centralizer of $M$ in $K$ must also be invariant
under conjugation by members of $K;$ and by definition that
centralizer will be invariant under conjugation by members of $M;$
hence it is invariant under conjugation by
all members of $MK=G,$ i.e., it is normal in~$G.$
So by~\eqref{d.K_has_no_nm},
\begin{equation}\begin{minipage}[c]{35pc}\label{d.triv_cntrlzr}
The centralizer of $M$ in $K$ is trivial.
\end{minipage}\end{equation}

Returning to what we proved earlier about $H,$ note that
if $f^{-1}(K) < H$ had nontrivial intersection with
$H_0,$ then this would centralize $H_1,$ hence its image
in $G$ would be a subgroup of $K,$ nontrivial by~\eqref{d.triv_cap},
that centralized the image of $H_1,$ which we saw in~\eqref{d.M1,M2}
contains $M,$ contradicting~\eqref{d.triv_cntrlzr}.
So $f^{-1}(K)$ has trivial intersection with $H_0;$ and
similarly with $H_1.$
In view of~\eqref{d.H12}, this
forces $f^{-1}(K)$ to be the graph of an isomorphism between subgroups
$K_0 < G_0$ and $K_1 < G_1,$ each isomorphic to $K,$ and we see
that these act on $M_0$ and $M_1$ as $K$ acts on $M.$
Hence the
subgroup $M_0\,K_0 < G_0$ (and likewise $M_1\,K_1 < G_1)$ is isomorphic
to $MK = G,$ again contradicting our assumption that $G$ is not
a subquotient of $G_0$ or $G_1.$
This gives the desired contradiction to~(ii) for
$G$ not Rhodes-prime with respect to direct products.

To prove Rhodes-primeness with respect to direct products
in case~(iii), suppose a subgroup $H$ of a finite group $G_0\times G_1$
can be mapped surjectively to $G=\Z_{p^n}.$
Then an element mapping to a generator of $G$ must
have order divisible by $p^n.$
Since the order of an element of $G_0\times G_1$ is the least
common multiple of the orders of its components, one of those
components must have order divisible by $p^n,$
hence some power of that element will have order exactly $p^n,$
hence the subgroup of $G_0$ or $G_1$ that it generates will
be a subquotient of $G_0$ or~$G_1$ isomorphic to $\Z_{p^n}=G.$
\end{proof}

(I obtained case~(ii) of the above result in 2014, answering a
question posed by John Rhodes, personal correspondence.
That result is given in weakened form in
\cite[Theorem~4.20, p.\,1275]{LRS}.
For the meaning of
``\!$G$ is ji''
in that statement, see the last paragraph of
\cite[p.\,1252]{LRS}, in particular, the display.)

Let us note some examples of the distinction between
Rhodes-primeness with respect to semidirect and direct products.

\begin{example}\label{E.S3}
In the category of finite groups, the following groups
are Rhodes-prime with respect to direct
products but, not being simple, are not
Rhodes-prime with respect to semidirect products.

\,\,{\rm(i)}\, The permutation groups $S_n$ for all $n\geq 3.$

\,{\rm(ii)}\, All semidirect products $\Z_p\rtimes A$ where
$p$ is a prime and $A$ a nontrivial subgroup of
$\mathrm{Aut}(\Z_p)\cong\Z_{p-1}.$

{\rm(iii)} All groups $\Z_{p^n}$ for $p$ prime and $n>1.$
\end{example}

\begin{proof}
To show $S_n$ Rhodes-prime with respect to direct
products, we use Proposition~\ref{P.x-prime}(ii).
For all cases except $n=4,$ $S_n$ is monolithic with monolith $A_n,$
and is a semidirect product of $A_n$ with the order-$\!2\!$
subgroup generated by any transposition.
The group $S_4$ has a different monolith,
the Klein four-group~$V,$ consisting of the even permutations of
exponent~$2,$ and if we write $S_3$ for the subgroup of elements of
$S_4$ fixing some one of the four elements on which $S_4$
acts, we find that $S_4$ is
a semidirect product $V\rtimes S_3,$ as required.
(For $n\geq 5,$ Rhodes-primeness of $S_n$
with respect to direct products is also an instance
of Proposition~\ref{P.x-prime}(i).)

In case~(ii) above, we again have a semidirect product decomposition.
To see that $\Z_p\rtimes A$ is monolithic with monolith $\Z_p,$
note that any $g\in\Z_p\rtimes A$ that is
not in $\Z_p$ acts nontrivially on $\Z_p,$
hence a commutator of $g$ with a nonidentity member of $\Z_p$ is a
nonidentity element of $\Z_p,$ so no
such $g$ can belong to a normal subgroup not containing~$\Z_p.$

In case~(iii), Rhodes-primeness with respect to direct products
is Proposition~\ref{P.x-prime}(iii).
\end{proof}

Each of the above examples shows that the first
implication of Lemma~\ref{L.=>=>} is not reversible.
I was unsure whether the second of those implications might be
reversible, but Alexander Olshanskiy provided Example~\ref{E.p^5} below,
showing that for every prime $p$ there is a group of
order $p^5$ for which that reverse implication fails,
and Example~\ref{E.Q8,D4},
showing that for $p=2,$ there are also two such examples of order~$p^3.$
Example~\ref{E.ZpZp^2} and the paragraph following it give
examples of order $p^3$ for all odd primes~$p.$
The referee observes that all these are what are called
``extraspecial $\!p$-groups'', namely, $\!p$-groups $G$
whose center is isomorphic to $\Z_p,$ and such that
the quotient of $G$ by its center has exponent $p$ \cite{wiki_extrsp}.

\begin{example}[A.\,Olshanskiy]\label{E.p^5}
Let $p$ be any prime, and $H$ the group {\rm(}of order $p^3)$
of upper triangular
$3\times 3$ matrices over $\Z_p$ with diagonal $I.$
Within $H,$ let $a=I+e_{23},$ $b=I+e_{12},$ $c=I+e_{13}.$
Let $G$ {\rm(}of order $p^5)$ be the quotient of $H\times H$ by the
subgroup generated by the central element $(c,\,c^{-1}).$

Then $G$ is monolithic, with monolith the subgroup $M$
generated by the image of $(c,I),$ which is also the image of $(I,c).$

By construction, $G$ is a quotient, hence a subquotient, of
$H\times H;$ but $G$ is not a subquotient of $H,$
since $H$ has smaller order.
Hence $G$ is not Rhodes-prime with respect to direct products
in the category of finite groups.
\end{example}

\begin{proof}
In $H,$ we find that $a^p=b^p=c^p=I,$ that $c$ is central,
and that $ba=abc.$
These properties allow us to write every element in the form
$a^i\,b^j\,c^k$ $(0\leq i,j,k<p),$ and since the total number of
such expressions is $p^3,$ the order of $H,$ this expression
for each element must be unique.
It is easy to verify that the subgroup generated by $c$ is
the commutator subgroup of $H,$ the center of $H,$
and a monolith in~$H.$

In $G,$ the subgroup $M$ described in the second
paragraph of the example is central, hence
normal, and has order $p,$ hence is minimal.
To see that it is the only minimal normal subgroup, note that
any $g\in G$ that is not in $M$ is the image of an
element $(h_0,h_1)\in H\times H$ such that at least
one of $h_0$ or $h_1$ involves a nonidentity power of $a$ or of $b.$
Assuming without loss of generality that $h_0$
involves a nonidentity power of $a,$ we find that the
commutator of $(h_0,h_1)$ with $(b,I)$ is a nonidentity
power of $(c,I).$
Hence in $G,$ the commutator of $g$ with
the image of $(b,I)$ is a generator of~$M.$
Thus, every normal subgroup of $G$ not contained in $M$ contains $M,$
so as $M$ is simple, $G$ is indeed monolithic with monolith~$M.$

The assertions of the final paragraph are immediate.
\end{proof}

\begin{example}[A.\,Olshanskiy]\label{E.Q8,D4}
Let us write $Q_8$ for the quaternion group
$\{\pm 1,\,\pm i,\,\pm j,\,\pm k\},$
and $D_4$ for the dihedral group {\rm(}the symmetry group of the square,
$<\!p,\,q\,\,|\,\,p^4=e,\,\,q^2=e,\,\,q\,p\,q^{-1}=p^{-1}\!\!>),$
both of order~$8.$

Then $Q_8$ is a subquotient of $D_4\times D_4,$
and $D_4$ a subquotient of $Q_8\times Q_8,$
but neither $Q_8$ nor $D_4$ is a subquotient of the other, hence
neither group is Rhodes-prime with respect to direct products
in the category of finite groups.

However, each is monolithic, with monoliths
$\{\pm 1\}\lhd Q_8$ and $\{e,\,p^2\}\lhd D_4$ respectively.
\end{example}

\begin{proof}
We shall show that the $\!16\!$-element group
\begin{equation}\begin{minipage}[c]{35pc}\label{d.16-elt}
$A\,=~<\!x,\,y\,\,|\,\,x^4=e,\,\,y^4=e,\,\,y\,x\,y^{-1}=x^{-1}\!>$
\end{minipage}\end{equation}
(a semidirect product
$<\!x\,|\,x^4=e\!>\rtimes<\!y\,\,|\,\,y^4=e\!>)$
is isomorphic both to a subgroup of $Q_8\times Q_8$
and to a subgroup of $D_4\times D_4,$ and
has both $Q_8$ and $D_4$ as homomorphic images,
from which the above subquotient assertions follow.

Within $Q_8\times Q_8,$ let $x=(i,1)$ and $y=(j,j).$
It is easy to check by looking at first coordinates,
and then at second coordinates, that
$x$ and $y$ satisfy the relations of~\eqref{d.16-elt},
hence the group they generate is a homomorphic
image of~$A.$
That group admits a homomorphism onto $Q_8,$ given
by projection to the first component;
but $x^2\,y^2=(1,-1)$ is in the kernel of that homomorphism,
hence the group generated by $x$ and $y$ must have
larger order than $Q_8;$ hence it can't be a {\em proper}
homomorphic image of the $\!16\!$-element
group ~$A,$ so it must be isomorphic thereto.

Likewise, within $D_4\times D_4,$
consider the subgroup generated by $x=(p,1)$ and $y=(q,p).$
As in the preceding paragraph, we verify that this
group is a homomorphic image of $A,$ and that
the projection onto the first component maps it surjectively to $D_4.$
However, $y^2=(e,p^2)$ is a nonidentity element
of the kernel of that homomorphism, so again the whole group is
of larger order than $Q_8,$ and so must be isomorphic to~$A.$

The final monolithicity assertions are easily verified.
\end{proof}

(We remark that $D_4$ is a semidirect
product, $<\!p\,\,|\,\,p^4=e\!>\,\rtimes<\!q\,\,|\,\,q^2=e\!>\!.$
However, $<\!p\,\,|\,\,p^4=e\!>$ is not the {\em least}
nontrivial normal subgroup $M$ of $D_4$ -- that is a subgroup
thereof -- so the above non-Rhodes-primeness result
does not contradict Proposition~\ref{P.x-prime}(ii).)

Here is one more example, which answers
a question raised by the referee:

\begin{example}\label{E.ZpZp^2}
For any odd prime $p,$ the monolithic group $G=\Z_{p^2}\rtimes\Z_p$
{\rm(}where a generator of $\Z_p$ acts on $\Z_{p^2}$ by
multiplication by $1\!+\!p,$ and the monolith is
the subgroup $p\,\Z_{p^2}\rtimes\{e\})$
is isomorphic to a subquotient of $\Z_{p^2}\times H$
where, as in Example~\ref{E.p^5}, $H$ is the group of upper triangular
$3\times 3$ matrices over $\Z_p$ with diagonal $I.$
Hence $G$ is not Rhodes-prime with respect to direct products
in the category of finite groups.
\end{example}

\begin{proof}[Sketch of proof]
In $\Z_{p^2}$ we will use additive notation,
writing $\Z_{p^2}=\{[0],[1],\dots,\,[p^2{-}1]\},$ while in $H$ we will
use the multiplicative notation of Example~\ref{E.p^5},
writing $H=\{a^i\,b^j\,c^k\ |\ 0\leq i,\,j,\,k <p\}$

Within $\Z_{p^2}\times H,$ take the subgroup of elements
such that the $\!\Z_{p^2}\!$-coordinate, and the exponent
of $a$ in the $\!H\!$-coordinate, agree modulo~$p;$
and divide this subgroup by the congruence equating the central
elements $([0],\,c)$ and $([p],\,I)$ (both of order $p).$

The resulting subquotient is a group of order $p^3$ in which the
image of $([1],\,b)\in\Z_{p^2}\times H$ generates
a subgroup of order $p^2,$ and the image of $([0],\,a)$
generates a subgroup of order $p,$ conjugation by which sends
the generator of the former subgroup to
its $\!1{+}p$-th power; so this subquotient is isomorphic to
$G=\Z_{p^2}\rtimes\Z_p.$
But $G$ is clearly not isomorphic to a subquotient
of either $\Z_{p^2}$ or $H.$
\end{proof}

Yet another example is the group denoted $H$ in
Examples~\ref{E.p^5} and~\ref{E.ZpZp^2} above.
That this monolithic group is not Rhodes-prime follows from
\cite[Theorem~5.2\,(ii), p.\,204]{DG}
(where $M$ denotes the group we are calling
$H,$ and $N$ the group called $G$ in Example~\ref{E.ZpZp^2} above).
Namely, the case $k=r=2$ of that statement says that the quotient
group of $G\times G$ obtained by identifying the monoliths
of the two factors is isomorphic to
the quotient group of $G\times H$ obtained in the same way.
Hence $H$ is a subquotient of $G\times G;$
but since $G$ and $H$ have the same order, $H$ is not
a subquotient of $G,$ proving it non-Rhodes-prime with
respect to direct products.

These examples lead one to wonder

\begin{question}\label{Q.p-gps}
For $p$ a prime, what finite $\!p\!$-groups {\em are} Rhodes-prime
with respect to direct products in the category of finite groups?
Are there any other than the cyclic $\!p\!$-groups $\Z_{p^n}$?
\end{question}

On the other hand, are there any finite monolithic groups not
Rhodes-prime with respect to direct products
that {\em aren't} $\!p\!$-groups?
Yes.
Modifying a construction suggested by the referee, let us consider
the variant of Example~\ref{E.p^5} gotten by taking $p>2,$
and letting $H$ consist of all upper triangular
$3\times 3$ matrices over $\Z_p$
in which $e_{11}$ and $e_{33}$ have coefficient~$1,$
but the coefficient of $e_{22}$ can be any nonzero member of $\Z_p;$
thus $|H|=(p{-}1)\,p^3.$
Obtain $G,$ as before, by dividing $H\times H$ by the normal
subgroup of elements $(I+fe_{1,3},\,I-fe_{1,3})$ for $f\in\Z_p;$
so $|G|=(p{-}1)^2\,p^5.$
The reader can verify that the image of the subgroup of
elements $(I+fe_{1,3},\,I)$ is again a monolith.
As before, though $G$ is a
subquotient of $H\times H,$ it is too large to be a subquotient of
$H,$ hence it is not Rhodes-prime with respect to direct products.
(One can generalize this construction in a way that allows
the monolith to have order $2$: Replace $\Z_p$ by any finite
field of cardinality $>2,$ in particular allowing finite
proper extensions of $\Z_2,$ and at the end of the construction,
divide the central subgroup of elements $(I+fe_{1,3},\,I+g\,e_{1,3})$
not just by the subgroup of elements $(I+fe_{1,3},\,I-fe_{1,3}),$ but by
a large enough overgroup thereof to shrink it down to a cyclic group,
so that it again becomes a monolith.)

Moving on, it is easy to give examples showing nonreversibility
of the last implication of Lemma~\ref{L.=>=>}.
E.g.:

\begin{example}\label{E.p,q1,q2}
Let $p$ be any prime, and $q_0,$ $q_1$ primes {\rm(}possibly
equal to one another{\rm)} that are both \mbox{$\equiv 1\pmod{p},$}
so that the automorphism groups of $\Z_{q_0}$ and $\Z_{q_1}$
both have orders divisible by $p.$
Thus, $\Z_p$ has faithful actions on both these groups.
Let $G=(\Z_{q_0}\times\Z_{q_1})\rtimes\Z_p,$ defined using
these actions.

Then $G$ is Tarski-prime in the category of finite groups, but
is not monolithic.
\end{example}

\begin{proof}
To show Tarski-primeness, note that
the order of $G,$ $p\,\,q_0\,q_1,$ is a product of just three
primes, hence if $G$ were a nontrivial direct product,
one of the factors would have prime
order, hence be commutative, and since it centralizes
the other factor, it would be central in $G.$
But $G$ has no nonidentity central elements; so
it is not such a product, so by the Krull-Schmidt Theorem,
$G$ is Tarski-prime in the category of finite groups.

But $G$ is not monolithic, since it has at least the two minimal
normal subgroups $\Z_{q_0}$ and $\Z_{q_1}.$
(These are its only minimal normal subgroups if $q_0\neq q_1.$
If $q_0=q_1,$ on the other hand, then $\Z_{q_0}\times\Z_{q_1}$ is a
$\!2\!$-dimensional $\Z_{q_0}\!$-vector space, and
has $q_0\!+\!1$ one-dimensional subspaces; and if we
give $\Z_p$ the same action on both of those
copies of $\Z_{q_0},$ then all $q_0\!+\!1$ of those subspaces
are minimal normal subgroups of~$G.)$
\end{proof}

The examples and proofs given above
make much use of subgroups,
but less of homomorphic images.
This led me to wonder:  Suppose we
call an object {\em modified} Rhodes-prime with
respect to direct  or semidirect products if it satisfies
condition~(iv) or~(v) of Convention~\ref{Cv.Rh&Ta} with
``subquotient'' replaced by ``subgroup''.
How do these conditions compare with the unmodified conditions?

These questions are easily answered for finite groups:

A finite group~$G$ is modified Rhodes-prime
with respect to direct products if and only if it is monolithic.
Indeed, if $G$ has a monolith~$M,$ then
for any embedding $G\to G_1\times G_2,$ the kernels of the induced maps
$G\to G_1,$ $G\to G_2,$ having trivial intersection,
cannot both contain~$M,$ so one of them must be trivial,
giving an embedding of $G$ in $G_1$ or $G_2,$
proving modified Rhodes-primeness.
On the other hand, any non-monolithic finite group $G$
has nontrivial normal subgroups
$N_1,\,N_2$ with trivial intersection, giving
an embedding $G\to G/N_1\times G/N_2,$
and looking at cardinalities, we see that $G$ cannot
be embedded in either $G/N_1$ or $G/N_2,$ so $G$
is not modified Rhodes-prime with respect to direct products.

On the other hand, we can see from the proofs of
Propositions~\ref{P.x|-prime} and~\ref{P.AO_simple} that modified
Rhodes-primeness of finite groups with respect to
{\em semidirect} products is equivalent to simplicity,
hence equivalent to ordinary Rhodes-primeness with respect to
semidirect products.

Note that the modified Rhodes-primeness conditions are
not, a priori, either stronger or weaker than the unmodified
conditions: \ %
The implication defining each of the modified conditions
has a strengthened hypothesis, that $G$ be isomorphic to a
{\em subgroup} of a group $G_0\times G_1$ or $G_0\rtimes G_1,$
not merely to a subquotient;
but it has a similarly strengthened conclusion.
For finite groups, the modified Rhodes-primeness conditions
turned out to be weaker than or equivalent to
the unmodified conditions, but for not-necessarily-finite
groups, or various classes thereof, the relationships
may be different.

If we look at the condition of modified Rhodes-primeness
with respect to direct products on
not necessarily finite groups, a notable class
of examples are the infinite abelian groups of prime exponent~$p.$
If such a group $G$ is embedded in a direct product $G_1\times G_2,$
then regarding $G$ as a $\!\Z_p\!$-vector space,
its images in $G_1$ and $G_2$ will
also be $\!\Z_p\!$-vector spaces, at least one of which must have
the same infinite dimension as~$G,$ hence $G$ can be embedded in
one of those groups (even if not via the given map to that group).
Thus $G$ is modified-Rhodes prime with respect to direct products.
But $G$ does not have a minimal normal subgroup, or even
satisfy the condition that the intersection
of any two nontrivial normal subgroups be nontrivial --
it has many pairs of nonzero $\!\Z_p\!$-subspaces
with zero intersection.
I have not examined whether $G$ also has the {\em unmodified}
Rhodes-primeness condition with respect to direct
products, and/or the modified or unmodified condition with
respect to semidirect products.

Returning to the unmodified conditions,
another question I have not thought hard about is

\begin{question}\label{Q.fin_vs_inf}
If a finite group $G$ is Rhodes-prime with respect to direct products
in the category of finite groups,
must it also have that property in the category of all groups?
{\rm(}The converse is clearly true.{\rm)}
\end{question}

Propositions~\ref{P.x|-prime} and~\ref{P.AO_simple}
show that the answer to the corresponding question for
Rhodes-primeness with respect to {\em semidirect} products
is~``yes'', and~\eqref{d.MMT_pm} shows the same for Tarski-primeness.

Some further observations:

The analog of the Krull-Schmidt theorem is not true for
finite {\em semigroups.}
To see this, first note that in the proof of
Lemma~\ref{L.null-kappa_x} above,
the case of finite $\kappa$ uses only finite semigroups, and so shows
that $\textup{Null}(\kappa)$ is not Tarski-prime among
finite semigroups.
Taking $\kappa$ such that $\kappa\!+\!1$ is a prime number,
$\textup{Null}(\kappa)$ also cannot be a nontrivial direct product.
For another example see \cite[Exercise~4, p.\,265]{MMT}.

But as we saw in \S\S~\ref{S.prep}-\ref{S.N_etc},
things become much better if we move from semigroups to monoids.
In particular, the J\'onsson-Tarski Theorem
(\cite{JT}, \cite[p.\,290]{MMT})
proves unique factorization for finite algebras
in a large class of varieties including the variety of monoids,
and hence Tarski-primeness of all nontrivial non-factoring
algebras in those varieties.
(Regarding the language used in~\cite{MMT}, see discussion
preceding and following~\eqref{d.MMT_pm} above.
That discussion concerns a generalization,
\cite[Exercise~5.15\,(10), p.\,300]{MMT}, of
the J\'onsson-Tarski Theorem to not necessarily
finite algebras whose ``centers'', in the sense used
there, are finite.)

These thoughts lead to our last question.

\begin{question}\label{Q.smgp,mnd,gp}
Understanding the two sorts of Rhodes-primeness to be defined
for semigroups and monoids as we define them
in Convention~\ref{Cv.Rh&Ta} for groups
{\rm(}indeed, they are so defined for semigroups in \cite{RS}{\rm)},
and likewise for the condition
of Tarski-primeness {\rm(}which is defined for general algebras
in~\cite{MMT}{\rm)}, the following
five implications are clear.
Are some or all of them reversible?

For any {\em finite} group $G,$
\begin{equation}\begin{minipage}[c]{35pc}\label{d.R_pm_smdir=>=>}
\hspace{2.4em}$G$ is Rhodes-prime with respect to semidirect products as
a finite {\em semigroup}\\
$\implies$ $G$ is Rhodes-prime with respect to semidirect products as
a finite {\em monoid}\\
$\implies$ $G$ is Rhodes-prime with respect to semidirect products as
a finite {\em group}.
\end{minipage}\end{equation}
\begin{equation}\begin{minipage}[c]{35pc}\label{d.R_pm_dir=>=>}
\hspace{2.4em}$G$ is Rhodes-prime with respect to direct products as
a finite {\em semigroup}\\
$\implies$ $G$ is Rhodes-prime with respect to direct products as
a finite {\em monoid}\\
$\implies$ $G$ is Rhodes-prime with respect to direct products as
a finite {\em group}.
\end{minipage}\end{equation}
\begin{equation}\begin{minipage}[c]{35pc}\label{d.T_pm=>=>}
\hspace{2.4em}$G$ is Tarski-prime as a finite {\em semigroup}\\
$\implies$ $G$ is Tarski-prime as a finite {\em monoid,}
equivalently, as a finite {\em group.}
\end{minipage}\end{equation}

Likewise, for arbitrary groups $G,$
are some or all of the corresponding implications
with ``finite'' everywhere deleted reversible?
\end{question}

One can see from the J\'onsson-Tarski Theorem mentioned above
that a finite group is Tarski-prime as a monoid
if and only if it is Tarski-prime as a group, which is
why~\eqref{d.T_pm=>=>} above shows fewer conditions
than the other two displays.
(But in the final sentence of Question~\ref{Q.smgp,mnd,gp}
the analog of~\eqref{d.T_pm=>=>} could be
expanded to a $\!3\!$-condition implication like the others.)

Incidentally, because a semigroup or monoid is not
in general isomorphic to its opposite, there are actually two
versions of the concept of {\em semidirect} product for these objects,
one based on ``action on the right'' and the other
on ``action on the left''.
But since the opposite of every semigroup
or monoid is still a semigroup or monoid, general results
about each of these constructions imply
the corresponding results about the other.
Cf.~\cite[p.\,24, first sentence of next-to-last paragraph]{RS}.

\section{An observation on Rhodes-primeness with respect to direct\\
products in arbitrary varieties of algebras}\label{S.Rprime&var}

The concept of Rhodes-primeness with respect to
direct products, as defined in Convention~\ref{Cv.Rh&Ta}(iv) above,
makes sense in any variety of algebras.
(In contrast, in a general variety there's no obvious
version of the concept of semidirect product of objects, hence
no obvious version of the concept of
Rhodes-primeness with respect to semidirect products.)
We end this note with a result on finite algebras Rhodes-prime
with respect to direct products in general varieties of algebras.

Given an object or family of objects $X$ in a variety $V,$
I will write $\textup{Var}(X)$ for the subvariety of $V$
generated by $X.$
We allow $V$ to have infinitely many operations, and/or to
have operations of infinite arities --
the finiteness of the algebras considered
will make those features unimportant.

Abandoning the convention made for semigroups at the beginning
of this note, we will understand every variety with no zeroary
operations to have an empty algebra.
In such a variety, this will be a subquotient of every object,
hence will be considered Rhodes-prime with respect to direct products.
We likewise put aside the restriction made in
Convention~\ref{Cv.Rh&Ta} of excluding $\!1\!$-element
algebras from the classes called Rhodes-prime.

Recall that an element $x$ of a lattice is called
{\em join prime} if $x\leq y\vee z$ implies
$x\leq y$ or $x\leq z.$

\begin{proposition}\label{P.Rprime&var}
Let $V$ be any variety of algebras \textup{(}in the sense
of universal algebra\textup{)} and $X$ a finite object of $V.$
Then the following conditions are equivalent:

\,\,{\rm(i)}\, $X$ is Rhodes-prime with respect to direct
products in the category of finite objects of $V.$

\,{\rm(ii)}\, If $Y_0$ and $Y_1$ are finite nonempty
objects of $V$ such that
$X\in\textup{Var}(\{Y_0,\,Y_1\}),$ then $X$ is a subquotient either
of $Y_0$ or of $Y_1.$

{\rm(iii)} $\textup{Var}(X)$ is join-prime in the join-semilattice of
subvarieties of $V$ generated by finite families
of finite algebras; and
whenever $Y$ is a finite nonempty algebra in $V$ such that
$\textup{Var}(X) \subseteq\textup{Var}(Y),$ $X$ is a subquotient of $Y.$
\end{proposition}

\begin{proof}

We shall prove (i)$\implies$(ii)$\implies$(iii)$\implies$(i).

Assuming (i), suppose we are given $Y_0,$ $Y_1$ as in
the hypothesis of~(ii).
Then taking $n$ such that $X$ is generated by $n$
elements, the assumption $X\in\textup{Var}(\{Y_0,\,Y_1\})$ makes
$X$ a homomorphic image of the free algebra on $n$ generators in
$\textup{Var}(\{Y_0,\,Y_1\}),$ which can be
constructed as a subalgebra of
$Y_0^{\textup{card}(Y_0)^n}\!\!\!\times Y_1^{\textup{card}(Y_1)^n}\!\!.$
This makes $X$ a
subquotient of a finite direct product of copies of $Y_0$ and $Y_1.$
Applying the Rhodes-primeness condition~(i) inductively, we
conclude that that $X$ is a subquotient of $Y_0$ or of $Y_1,$
proving~(ii).

Now assume~(ii).
Since any subvariety of $V$ generated by finitely many finite
algebras is generated by one finite algebra (the direct product of the
nonempty algebras in the family), the first assertion of~(iii) is
equivalent to saying
that if $\textup{Var}(X)\subseteq\textup{Var}(\{Y_0,\,Y_1\}),$
then $\textup{Var}(X)\subseteq\textup{Var}(Y_0)$
or $\textup{Var}(X)\subseteq\textup{Var}(Y_1).$
This condition is trivial if one of $Y_0,$ $Y_1$ is empty,
while assuming them nonempty, it implied by the subquotient
conclusion of~(ii).
The second assertion of~(iii) is obtained from~(ii) by
taking $Y_0=Y_1=Y.$

Finally, assuming~(iii), suppose $X$ is a subquotient of
$Y_0\times Y_1.$
Then in particular, it belongs to the variety
generated by $Y_0$ and $Y_1,$
so by the first condition of~(iii), $X$ belongs
to the variety generated by one of these, which we will call $Y.$
Then assuming $Y$ nonempty, the second condition of (iii) shows
that $X$ is a subquotient of $Y,$ establishing~(i).
On the other hand, if $Y$ is empty, then so is~$X,$
and since the empty algebra (if it exists in $V)$
is Rhodes-prime with respect to direct products, we again get~(i).
\end{proof}

We note that neither of the two conditions of (iii) implies the other.
Indeed, let $V$ be the variety of abelian groups.
Note that every subvariety of $V$ is determined by
an identity $x^n=e$ for a unique nonnegative integer~$n.$
It is easy to see that for any prime $p,$ the group $X=\Z_p\times\Z_p$
satisfies the first part of~(iii) but (taking $Y=\Z_p),$
not the second, while for any two distinct primes $p\neq q,$ the group
$X=\Z_p\times\Z_q$ satisfies the second but not the first.

In that same condition~(iii), can the inclusion
$\textup{Var}(X) \subseteq\textup{Var}(Y)$ be replaced by
$\textup{Var}(X)=\textup{Var}(Y)$?
Here is a version of an example provided by the referee
showing that it cannot.
Let $G$ be a nontrivial finite group, $V$ the variety
of $\!G\!$-sets, and $X$ a $\!2\!$-element set on which
$G$ acts trivially.
Thus, $X$ is a subquotient of a $\!G\!$-set $Y$ if and
only if $Y$ has more than one orbit.
In particular, if $Y$ is a free $\!G\!$-set on one generator,
then $X$ is not a subquotient of $Y,$ but is
a subquotient of $Y\times Y;$ so $X$ is not
Rhodes-prime with respect to direct products.
But turning to condition~(iii), we see that
$\textup{Var}(X)$ is the variety of $\!G\!$-sets
on which $G$ acts trivially, which is a minimal
nontrivial element of the lattice of subvarieties
of $V,$ hence join-prime; and we easily see that the
``whenever'' condition of~(iii) with $\subseteq$ replaced by~$=$ also
holds for~$X;$ so that modified version of~(iii) is not
equivalent to~(i).

\section{Acknowledgements}\label{S.ackn}
I am grateful to Jonathan Farley for his group-email pointing
out that the result of~\cite{YCor} answers one of the questions
listed at~\cite[p.\,263]{MMT}, which drew my attention to that list of
questions, in particular the one now answered
in \S\ref{S.Main} above; to Ralph Freese
and Walter Taylor for information about what is known about
Tarski-primeness; to Benjamin Steinberg,
Stuart Margolis
and Edmond Lee for information about what is known
about the two Rhodes-primeness conditions; to Alexander Olshanskiy
for Proposition~\ref{P.AO_simple}, Examples~\ref{E.p^5}
and~\ref{E.Q8,D4}, and some pointers on terminology, and
to the referee for further information about known results
and terminology, and some useful examples and suggestions.

\end{document}